\documentclass{article}

\usepackage{arxiv}

\usepackage[utf8]{inputenc} 
\usepackage[T1]{fontenc}    
\usepackage{hyperref}       
\usepackage{url}            
\usepackage{booktabs}       
\usepackage{amsfonts}       
\usepackage{nicefrac}       
\usepackage{microtype}      
\usepackage{graphicx}

  \usepackage{calrsfs}
  \usepackage{enumerate}
  \usepackage[shortlabels]{enumitem}
  \usepackage{longtable}
  \usepackage{amsmath,amssymb,amsfonts,amsthm}
  \usepackage{algorithm}
  \usepackage{comment}
  \usepackage{xcolor}
  \usepackage[lofdepth,lotdepth]{subfig}
  \usepackage{pgf,tikz}
  \usepackage{pgfplots}

  \DeclareMathAlphabet\mathpzc{T1}{pzc}{mb}{it}
  \DeclareMathAlphabet{\pazocal}{OMS}{cmsy}{m}{n}

  \def\norml{\left\|}
  \def\normr{\right\|}
  
  \DeclareMathOperator{\divv}{div}
  \DeclareMathOperator{\Divv}{\textbf{div}}
  \DeclareMathOperator{\gradd}{\nabla\!}
  \DeclareMathOperator{\grad}{\nabla\!}
  \DeclareMathOperator{\maxx}{p_{+}}

  \DeclareMathOperator{\normalInt}{\mathbf{n}_\mathbf{o}}
  \DeclareMathOperator{\tanInt}{\mathbf{t}_\mathbf{o}}
  \DeclareMathOperator{\normalExt}{\mathbf{n}}
  \DeclareMathOperator{\tanExt}{\mathbf{t}}

  \def\prodL2#1#2{\left(#1\right)_{#2}}
  
  \def\prodD#1#2{\langle#1\rangle_{#2}}

  \DeclareMathOperator{\pp}{\mathbf{p}}
  \DeclareMathOperator{\qq}{\mathbf{q}}
  \DeclareMathOperator{\uu}{\mathbf{u}}
  \DeclareMathOperator{\vv}{\mathbf{v}}
  \DeclareMathOperator{\ww}{\mathbf{w}}
  
  \DeclareMathOperator{\yy}{\mathbf{y}}
  \DeclareMathOperator{\gG}{\mathbf{g}}
  \DeclareMathOperator{\Aa}{\mathbb{C}}
  \DeclareMathOperator{\Bb}{\mathbf{B}}
  \DeclareMathOperator{\Kk}{\mathbf{K}}
  \DeclareMathOperator{\Ll}{\mathbf{L}}
  \DeclareMathOperator{\Hh}{\mathbf{H}}
  
  \DeclareMathOperator{\Cc}{\pmb{\pazocal{C}}}
  \DeclareMathOperator{\Ii}{\mathbf{I}}
  \DeclareMathOperator{\Id}{\mathrm{Id}}
  \DeclareMathOperator{\ff}{\mathbf{f}}
  \DeclareMathOperator{\tauu}{\boldsymbol{\tau}}
  \DeclareMathOperator{\epsilonn}{\boldsymbol{\epsilon}}
  
  \DeclareMathOperator{\sigmaa}{\boldsymbol{\sigma}}
  \DeclareMathOperator{\muu}{\boldsymbol{\mu}}
  \DeclareMathOperator{\nuu}{\boldsymbol{\nu}}
  
  \DeclareMathOperator{\thetaa}{\boldsymbol{\theta}}
  
  \DeclareMathOperator{\xii}{\boldsymbol{\xi}}

  \DeclareMathOperator{\Xx}{\mathbf{X}}

  \definecolor{dartmouthgreen}{rgb}{0.05, 0.5, 0.06}

\newtheorem{theorem}{Theorem}

\newtheorem{lemma}{Lemma}

\newtheorem{remark}{Remark}
\newtheorem{assumption}{Assumption}

\title{A shape optimization algorithm based on directional derivatives for three-dimensional contact problems}

\author{
 Bastien Chaudet-Dumas \\
  Section of Mathematics\\
  University of Geneva\\
  \texttt{bastien.chaudet@unige.ch} \\
}

\begin{document}
\maketitle
\begin{abstract}
    This work deals with shape optimization for contact mechanics. More specifically, the linear elasticity model is considered under the small deformations hypothesis, and the elastic body is assumed to be in contact (sliding or with Tresca friction) with a rigid foundation. The mathematical formulations studied are two regularized versions of the original variational inequality: the penalty formulation and the augmented Lagrangian formulation. In order to get the shape derivatives associated to those two non-differentiable formulations, we follow an approach based on directional derivatives \cite{chaudet2020shape,chaudet2021shape}. This allows us to develop a gradient-based topology optimization algorithm, built on these derivatives and a level-set representation of shapes. The algorithm also benefits from a mesh-cutting technique, which gives an explicit representation of the shape at each iteration, and enables us to apply the boundary conditions strongly on the contact zone. The different steps of the method are detailed. Then, to validate the approach, some numerical results on two-dimensional and three-dimensional benchmarks are presented. 
\end{abstract}


\section{Introduction}

Structural optimization has become an integral part of industrial conception, with applications in more and more challenging mechanical contexts. Those mechanical contexts often lead to complex mathematical formulations involving non-linearities and/or non-differentiabilities, which causes many difficulties when considering the associated shape optimization problem. This is the reason why this topic has been widely studied for the past fifty years in various research fields such as engineering, optimization and numerical analysis.
In this article, we are interested in finding the optimal shape of a body in contact with a rigid foundation, in the sense that this shape must maximize or minimize some given mechanical criterion. The deformable body is assumed to be made of a linear elastic material and the contact model considered is either frictionless (sliding) or with Tresca friction. 

The structural optimization problem takes the form of a constrained optimization problem in infinite dimension, one of the constraints being the state equation modeling the mechanical equilibrium (usually written as a variational formulation). In general, the numerical resolution of this optimization problem is tackled using gradient based optimization algorithms, which of course require the expressions of the shape derivatives. There exist two main families of approaches to address this issue. Following the usual terminology \cite{LiPet2004,heinemann2016shape}, we first mention the paradigm \textit{optimize-then-discretize}, which consists in writing the optimality conditions associated to the state equation in the infinite dimensional setting, then discretizing this state equation and the optimality conditions obtained \cite{allaire2004structural}. The second paradigm, referred to as \textit{discretize-then-optimize}, consists in discretizing the state equation first, then write the related shape optimization problem in finite dimension and derive the associated optimality conditions \cite{BenSig2003}. Here, we follow the first one, and we compute our shape derivatives using Hadamard's boundary variation method. Note that, since the pioneer work of Hadamard \cite{hadamard1908memoire}, there have been many studies following this approach \cite{murat1975etude,simon1980differentiation,pironneau1982optimal,sokolowski1992introduction,DelZol2001,henrot2006variation}. 

Another key issue in shape optimization is the representation of the shapes. In this article, we use a representation based on the level-set method, as it was first exposed in the 2000's \cite{allaire2004structural} in the context of structural optimization. The use of a level-set function $\phi$ is very interesting in practice as it enables us to deal with geometric deformations of the domain by simply solving an advection equation, and it also allows topology changes. However, when using this technique, as in the case of density methods\cite{BenSig2003}, we do not have a sharp representation of the interface $\partial\Omega=\{ \phi=0 \}$ \cite{van2013level}. Originally, the authors got around this difficulty by enforcing the boundary conditions approximately on $\partial\Omega$, by means of a regularized Heaviside function. In the same spirit, we also mention shape optimization algorithms that use immersed boundary techniques which are very efficient to handle boundary conditions when no mesh of the boundary is available. Among these algorithms, some use the very popular XFEM method\cite{belytschko2003topology,kreissl2012levelset}. Here, we choose a different approach which consists in coupling the level-set method with a conformal discretization. This means that we work with a description of the shape which is not only implicit (by means of $\phi$) but also explicit (by means of an actual mesh for the discrete shape). There have been various works in this direction  \cite{ha2008level,yamasaki2011level,allaire2011topology,allaire2014shape}, some of which even use a boundary element method to solve the mechanical problem \cite{abe2007boundary}. Obviously, this approach is more expensive as it involves some remeshing procedure at each shape optimization iteration. However, it enables us to limitate the numerical resolution of the mechanical problem to $\Omega$ only, and to apply exactly the boundary conditions on $\partial\Omega$, which is crucial in the case of contact mechanics as the quality of the obtained solution highly depends on the accuracy in the neighbourhood of the contact zone. In order to build this explicit representation of $\Omega$, we propose a new mesh cutting technique based on a quadratic interpolation of the level-set function, which is very simple, not very costly, and gives a rather smooth  representation of the shapes.

Given the assumptions made on the material and on the friction model, the mathematical formulation associated to our contact problem takes the form of a variational inqueality of the second kind, for which existence, uniqueness and regularity of the solution have been established \cite{DuvLio1972,boieri1987existence}. In the context of shape optimization, this formulation is not really convenient as it is not differentiable with respect to the shape because of the projection operators involved in the expression of the contact boundary conditions. This issue may be addressed in different ways. First, it is possible to introduce a weaker notion of differentiability, called conical differentiabilty \cite{mignot1984optimal}, and work with this notion to derive the associated first order optimality conditions. This approach has been applied to our problem in two dimensions \cite{sokolowski1988shape}, but the expressions of the shape derivatives are not very usable in numerical practice. Second, following the \textit{discretize-then-optimize} paradigm, it is possible to discretize the formulation, then use the tools from subdifferential calculus. This approach has been successfully applied in the context of shape optimization for elastic bodies in frictional contact with a plane \cite{kovccvara1994optimization,beremlijski2002shape,haslinger2012shape,beremlijski2014shape}.
Finally, a popular approach among the mechanical engineering community is to consider a regularized but approximate formulation of the original variational inequality such as the penalty formulation or the augmented lagrangian formulation. This new formulation usually takes the form of a non-linear non-differentiable variational formulation. Regularizing all non-smooth functions leads to an approximate formulation that is classically shape differentiable. There have been a few works in this direction using the penalty method \cite{kim2000meshless}, even recently \cite{maury2017shape}. In this last reference, the authors use the level-set method and they compute shape derivatives for the continuous problem in two and three dimensions, but do not take into account a possible gap between the bodies in contact. As for the augmented lagrangian method, there exists a mathematical analysis of shape sensitivity together with the associated numerical results for some specific cost functionals in two dimensions \cite{fancello1994shape,fancello1995numerical}, but these studies are limited to parametric shape optimization (using B-splines) in the frictionless case. Finally, let us also mention a very recent work \cite{bretin2022shape} where the authors consider the geometric shape optimization of a rolling structure in frictional contact with a plane in two and three dimensions, without optimizing the contact zone. The contact is modeled using Nitsche's method \cite{chouly2013nitsche}, and they use a level-set representation of the shapes, with a fictitious domain approach close to XFEM to solve the mechanical formulation.

In the present work, we propose a shape optimization method that optimizes the whole boundary of the domain (including the contact zone) as well as its topology in the three-dimensional frictional case. Moreover, we consider the general case of a structure in contact with a non necessarily plane rigid object, and we take into account a possible gap between the two bodies. The contact formulations studied are the penalty formulation and the augmented lagrangian formulation, without any additional regularization procedure. Our approach relies on an original and rigorous computation of the shapes derivatives based on the directional derivatives of the non-smooth functions that appear in these formulations \cite{chaudet2020shape,chaudet2021shape}. Furthermore, thanks to our novel mesh-cutting technique, the numerical method presented benefits from an explicit representation of the shape without spending too much computational time on the remeshing procedure. This enables us to enforce the contact condition exactly on the boundary. 

This article is structured as follows. Section 2 presents the mechanical problem and its different formulations: the original formulation and the two regularized formulations studied (penalty and augmented lagrangian). In section 3 we discuss the shape optimization problem, in particular we derive the shape derivatives associated to each of the two regularized formulations. Section 4 gives a detailed description of the numerical shape optimization algorithm, in which the mesh-cutting technique is presented and discussed. Finally, some numerical results are presented in section 5. These results help validate the approach and allow for comparison between the penalty method and the augmented lagrangian method.

\section{Mechanical problem formulation}

We are considering in $\mathbb{R}^d$, $d \in \{2,3\}$, the problem of a deformable body coming in contact with a rigid foundation under the action of body forces and surface loads. We denote by $\Omega$ the domain representing the deformable body and by $\Omega_{rig}$ the rigid foundation. Let $\Gamma_D$ be the part of the boundary where a homogeneous Dirichlet condition applies (blue part), $\Gamma_N$ the part where a non-homogeneous Neumann condition $\tauu$ applies (orange part), $\Gamma_C$ the potential contact zone (green part), and $\Gamma$ the rest of the boundary, which is free of any constraint (i.e. homogeneous Neumann boundary condition). Those four parts are mutually disjoint and we have: $\overline{\Gamma_D} \cup \overline{\Gamma_N} \cup \overline{\Gamma_C} \cup \overline{\Gamma} = \partial \Omega$, as depicted in Figure \ref{fig:SchOpen}.
The outward normal to $\Omega$ is denoted by $\normalInt$. Similarly, the inward normal to $\Omega_{rig}$ is denoted by $\normalExt$.
We also introduce the space $\Xx:=\Hh^1_{\Gamma_D}(\Omega)=\{ \vv \in (H^1(\Omega))^d \: | \: \vv|_{\Gamma_D} = 0 \}$, and $\Xx^*$ its dual.
Finally, for any $v$ vector in $\mathbb{R}^d$, the product with the normal $v \cdot \normalInt$ (respectively with the normal to the rigid foundation $v \cdot \normalExt$) is denoted by $v_{\normalInt}$ (respectively $v_{\normalExt}$). Similarly, the tangential part of  $v$ is denoted by $v_{\tanInt} = v - v_{\normalInt} \normalInt$ (respectively $v_{\tanExt} = v - v_{\normalExt} \normalExt$).
Since we are interested in small deformations, the body $\Omega$ is assumed to be made of a linear elastic material and we consider the small displacements assumption.
Therefore, if the physical displacement is denoted by $\uu \in \Xx$, then the stress tensor is given by Hooke's law :
$$
	\sigmaa(\uu) = \Aa : \epsilonn(\uu) = 2\bar{\mu} \epsilonn(\uu) + \bar{\lambda}\divv \uu \Ii\:,
$$
where $\boldsymbol{\epsilon}(\uu)=\frac{1}{2}(\gradd\uu+\gradd\uu^T)$ is the linearized strain tensor, $\Aa$ is the elasticity tensor and $\Ii$ is the identity tensor of order $2$. The coefficients $\bar{\mu}$ and $\bar{\lambda}$ (referred to as the Lamé coefficients) are such that $\Aa$ is elliptic. Regarding external forces, the body force $\ff \in \Ll^2(\Omega)$, and the traction (or surface load) $\tauu \in \Ll^2(\Gamma_N)$. 

\begin{remark}
	The Lamé coefficients are often expressed in terms of the Young modulus $E$ and the Poisson coefficient $\nu$ by means of the following formulae:
	\begin{equation*}
    	\bar{\lambda} = \frac{E\nu}{(1+\nu)(1-2\nu)}\:, \hspace{1em} \bar{\mu} = \frac{E}{2(1+\nu)}\:.
	\end{equation*}
\end{remark}

\begin{figure}
\begin{center}
\begin{tikzpicture}

\fill [blue!60] plot [smooth cycle] 
coordinates {(3.9,1.8) (4.055,1.64) (4.095,1.5) (4.055,1.34) (3.73,1.175) (3.76,1.5)};

\draw[black,fill=blue!15, thin,densely dotted,rotate=-15] (3.3,2.425) ellipse (0.05cm and .33cm);

\fill [orange!50] plot [smooth  cycle] 
coordinates {(0.2,1.22) (0.3,1.6) (0.05,1.22) (0.03,0.9) (0.15, 1.)};

\draw[black,fill=orange!15, thin,densely dotted,rotate=-23] (-0.3,1.22) ellipse (0.05cm and .42cm);

\fill [dartmouthgreen!60] plot [smooth  cycle] 
coordinates {(2.795,.50) (2.87,.1) (2.65,-.55) (1.5, -.55) (.5,-.3) (0.15,.5) (1.5, .6)};

\draw[black, fill=dartmouthgreen!25, thin,dotted, rotate=-0.5] (1.51,0.485) ellipse (1.4cm and .12cm);

\draw [black] plot [smooth cycle] 
coordinates {(0,1) (0.3, 1.6) (1,2) (2, 2.1) (3,2) (3.9,1.8) (4.055,1.64) (4.095,1.5) (4.055,1.34) (3.73,1.175) (3,1) (2.7,-.5) (1,-.5) (0.3, -.1)};

\draw[black, fill=dartmouthgreen!60,thin,dashed,rotate=-.5] (1.6,-.39) ellipse (.9cm and .1cm);

\draw[black, ultra thin] (-2,-0.75) -- (2.75,-0.75) -- (4.75,-.25) -- (0, -.25) -- (-2,-0.75);

\node[] at (2.9,2.3) {$\Gamma$};
\node[] at (4.5,1.25) {$\Gamma_D$};
\node[] at (-.25,1.5) {$\Gamma_N$};
\node[] at (3.3,0.) {$\Gamma_C$};
\node[] at (1.7,1.25) {$\Omega$};
\draw[->] (0.3, 0.1) -- (-0.1, -0.1) ;
\node[] at (0.5, 0.1) {$\mathrm{x}$};
\node[] at (-0.3, 0.2) {$\normalInt(\mathrm{x})$};
\draw[red, densely dashed] (0.3, 0.1) -- (0.3, -0.6) ;
\node[red] at (0.3,-0.6) {\tiny{$\bullet$}};
\node[red] at (-0.3, -0.4) {$\gG_{\normalExt}(\mathrm{x})$};
\draw[->] (0.3 , -0.6) -- (0.3, -1.1) ;
\node[] at (.75, -0.97) {$\normalExt(\mathrm{x})$};

\end{tikzpicture}
\end{center}
  \caption{Elastic body in contact with a rigid foundation.}
  \label{fig:SchOpen}
\end{figure}
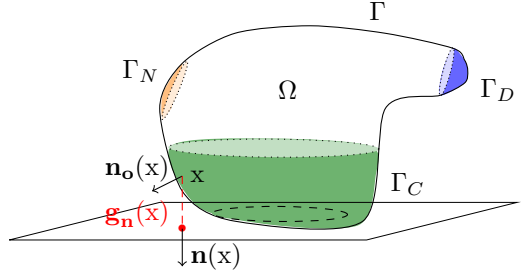

\subsection{Contact boundary conditions} 

The two essential ingredients to model the contact phenomenon are a non-penetration condition, which ensures that $\Omega$ does not penetrate $\Omega_{rig}$, and some friction model to take into account possible frictional effects.

\paragraph{Non-penetration condition}
At each point $x$ of $\Gamma_C$, let us define the gap $\gG_{\normalExt}(\mathrm{x})$, as the oriented distance function to $\Omega_{rig}$ at $x$, see Figure \ref{fig:SchOpen}.
The non-penetration condition can be stated as: $\uu_{\normalExt}\leq \gG_{\normalExt}$ a.e.$\!$ on $\Gamma_C$. This condition results in two inequalities and one equality holding a.e. on the potential contact zone $\Gamma_C$ :
\begin{equation}
  \left\{
	\hspace{0.5em}    
    \begin{aligned}
    	\uu_{\normalExt}\leq \gG_{\normalExt}\:, \\
    	\sigmaa_{\normalInt\!\normalExt}(\uu) \leq 0\:, \\
    	\sigmaa_{\normalInt\!\normalExt}(\uu)(\uu_{\normalExt}-\gG_{\normalExt})=0 \:,
    	\label{CLNP}
    \end{aligned}
  \right.    	
\end{equation}
where $\sigmaa_{\normalInt\!\normalExt}(\uu) = \sigmaa(\uu)\cdot\normalInt\cdot\normalExt$ is the normal constraint on $\Gamma_C$.
Besides, we also introduce the set of admissible displacements in a standard way \cite{eck2005unilateral}:
\begin{displaymath}
	\Kk:=\{\vv \in \Xx \: : \: \vv_{\normalExt}\leq \gG_{\normalExt} \:\:\mbox{a.e. on}\: \Gamma_C\} .
\end{displaymath}

\begin{remark}
  In general, the set of boundary conditions coming from the non-penetration condition reads
  \begin{equation}
       \uu_{\normalInt} \leq \gG_{\normalInt}, \hspace{1em}\sigmaa_{\normalInt\! \normalInt}(\uu) \leq 0, \hspace{1em}\sigmaa_{\normalInt\! \normalInt}(\uu)(\uu_{\normalInt}-\gG_{\normalInt})=0 \: \mbox{ on } \Gamma_C\:,
    \label{CL0}
  \end{equation}
  where $\gG_{\normalInt}(\mathrm{x})$ denotes the distance between $\mathrm{x} \in \Gamma_C$ and the rigid foundation computed in the direction of the normal $\normalInt$ to $\Gamma_C$.  However, since we assume that the deformable body undergoes small displacements relatively to its reference configuration, this set of conditions is equivalent to \eqref{CLNP}. 
  Indeed, under the small displacement hypothesis, one may replace the normal vector $\normalExt$ and the gap $\gG_{\normalExt}$ to the rigid foundation by $\normalInt$ and $\gG_{\normalInt}$ \cite{kikuchi1988contact}.

  In our context, we choose to write the formulation using conditions \eqref{CLNP} instead of conditions \eqref{CL0}. As it has been explained, this choice does not affect the solution to the contact problem. Moreover, we will see in the next sections that this slightly different formulation is very well suited to shape optimization.
  \label{RemHypHPP}
\end{remark}

\paragraph{Friction condition}
In this work, we consider the Tresca model to simulate frictional effects \cite{lions1967variational}. Even though it does not lead to a physical representation as accurate as in the case of the more realistic Coulomb friction model, this model has the advantage of having nicer mathematical properties, which are extensively used in the derivation of our shape optimization algorithm \cite{chaudet2020shape}. Let $\mathfrak{F} : \Gamma_C \rightarrow \mathbb{R}$, be the friction coefficient. In order to avoid regularity issues, the function $\mathfrak{F}$ is assumed to be uniformly Lipschitz continuous and strictly positive on $\Gamma_C$. The principle of the Tresca model is to replace the usual Coulomb threshold $|\sigmaa_{\normalInt \!\normalExt}(\uu)|$ by a fixed strictly positive function $s\in L^2(\Gamma_C)$, which leads to the following conditions on $\Gamma_C$:
\begin{equation}
  \left\{
	\hspace{0.5em}    
    \begin{aligned}
    	|\sigmaa_{\normalInt\!\tanExt}(\uu)| \:&<\: \mathfrak{F} s & \:\: \mbox{on } \{ \mathrm{x} \in \Gamma_C : \uu_{\tanExt}(\mathrm{x}) = 0 \}\:, \\
    	\sigmaa_{\normalInt\!\tanExt}(\uu) \:&=\: -\mathfrak{F} s  \frac{\uu_{\tanExt}}{|\uu_{\tanExt}|} & \:\: \mbox{on } \{ \mathrm{x} \in \Gamma_C : \uu_{\tanExt}(\mathrm{x}) \neq 0 \}\:,
    	\label{CL1}
    \end{aligned}
  \right.    	
\end{equation}
where $\sigmaa_{\normalInt\!\tanExt}(\uu) = \sigmaa(\uu)\cdot\normalInt-\sigmaa_{\normalInt\!\normalExt}(\uu)\normalExt$ is the tangential constraint on $\Gamma_C$. The sets $\{ \mathrm{x} \in \Gamma_C : \uu_{\tanExt}(\mathrm{x}) = 0 \}$  and $\{ \mathrm{x} \in \Gamma_C : \uu_{\tanExt}(\mathrm{x}) \neq 0 \}$ represent respectively \textit{sticking} and \textit{sliding} points. Of course, this simplification of the original Coulomb model leads to an approximate and sometimes inaccurate representation of frictional effects.

\subsection{Mathematical formulations of the problem}

Since the data has been chosen regular enough, it is known \cite{lions1967variational} that solving the contact problem is equivalent to finding the displacement $\uu$ solution to the strong formulation: 
\begin{subequations} \label{FF:all}
    \begin{align}
    - \Divv \sigmaa(\uu) &= \ff & \mbox{in } \Omega, \label{FF:1}\\
      \uu                &= 0   & \mbox{on } \Gamma_D,  \label{FF:2}\\
    \sigmaa(\uu) \cdot \normalInt &= \tauu & \mbox{on } \Gamma_N,   \label{FF:3}\\
    \sigmaa(\uu) \cdot \normalInt &= 0 & \mbox{on } \Gamma,   \label{FF:4}\\
     \uu_{\normalExt} \leq \gG_{\normalExt}, 
     \sigmaa_{\normalInt\!\normalExt}(\uu) &\leq 0, \sigmaa_{\normalInt\!\normalExt}(\uu)(\uu_{\normalExt}-\gG_{\normalExt})=0 & \mbox{on } \Gamma_C,   \label{FF:5}\\
    	|\sigmaa_{\normalInt\!\tanExt}(\uu)| \:&<\: \mathfrak{F} s & \:\: \mbox{on } \{ \mathrm{x} \in \Gamma_C : \uu_{\tanExt}(\mathrm{x}) = 0 \}\:, \label{FF:6}\\
    	\sigmaa_{\normalInt\!\tanExt}(\uu) \:&=\: -\mathfrak{F} s  \frac{\uu_{\tanExt}}{|\uu_{\tanExt}|} & \:\: \mbox{on } \{ \mathrm{x} \in \Gamma_C : \uu_{\tanExt}(\mathrm{x}) \neq 0 \}\:. \label{FF:7}
    \end{align}
\end{subequations}
For theoretical and numerical purposes, it is often convenient to work with a weak formulation equivalent to \eqref{FF:all}. In order to write this weak formulation, let us introduce the bilinear and linear forms $a : \Xx \times \Xx \rightarrow \mathbb{R}$ and $L : \Xx \rightarrow \mathbb{R}$, such that:
\begin{equation*}
  a(\uu,\vv) := \int_\Omega \Aa:\epsilonn(\uu):\epsilonn(\vv) \, \mathrm{dx} \:, \hspace{1.5em} L(\vv) := \int_\Omega \ff \vv \mathrm{dx} + \int_{\Gamma_N} \tauu \vv \mathrm{ds}\:.
\end{equation*}
Due to the inequalities enforced on the potential contact zone $\Gamma_C$, deriving the variational formulation associated to \eqref{FF:all} using the standard approach leads to a variational inequality (of the second kind):
\begin{equation}
	a(\uu,\vv-\uu) + j_T(\vv) - j_T(\uu) \:\geq\: L(\vv-\uu)\:, \:\:\: \forall \vv \in \Kk\: ,
    \label{IV}
\end{equation}
where the non-linear functional $j_T : \Xx \rightarrow \mathbb{R} $ is defined by:
\begin{equation*}
	j_T(\vv) := \int_{\Gamma_C} \mathfrak{F} s |\vv_{\tanExt}| \,\mathrm{ds}\: .
\end{equation*}
Note that this problem can also be reformulated as a minimization problem: find $\uu\in\Kk$ solution to
\begin{equation}
    \inf_{\vv\in\Kk} \ \frac{1}{2}a(\vv,\vv) - L(\vv) + j_T(\vv)\:.
    \label{OPT}
\end{equation}
Existence and uniqueness of the solution $\uu\in \Kk$ have already been established, as well as equivalence between formulations \eqref{FF:all}, \eqref{IV} and \eqref{OPT}.
The interested reader may find further details in the literature \cite{DuvLio1972,oden1980theory,Cia1988}.

Using the tools from convex analysis and calculus of variations \cite{ekeland1999convex,ito2008lagrange}, it is possible to obtain another formulation of the contact problem based on \textit{Lagrange multipliers}.
More specifically, one can prove \cite{stadler2004infinite} that, if $\uu \in \Kk$ solves \eqref{OPT}, then there exists a unique pair of dual variables $(\lambda,\muu) \in H^{-\frac{1}{2}}(\Gamma_C) \times \Hh^{-\frac{1}{2}}(\Gamma_C)$ such that:
\begin{subequations} \label{LMF:all}
    \begin{align}
    	a(\uu,\vv) - L(\vv) + \prodD{\lambda, \vv_{\normalExt}}{\Gamma_C} + \prodD{\muu, \vv_{\tanExt}}{\Gamma_C} &= 0\:, \hspace{1em} \forall \vv \in \Xx\:, \label{LMF:1}\\
    	\prodD{\lambda, \zeta }{\Gamma_C}  &\geq 0\:, \hspace{1em} \forall \:\zeta \in H^\frac{1}{2}(\Gamma_C)\:, \:\: \zeta \geq 0\:, \label{LMF:2}\\
    	\prodD{\lambda, \uu_{\normalExt}-\gG_{\normalExt}}{\Gamma_C} &= 0\:, \label{LMF:3}\\
    	\prodD{ \mathfrak{F} s, |\nuu| }{\Gamma_C} - \prodD{ \muu, \nuu }{\Gamma_C} &\geq 0\:, \hspace{1em} \forall \nuu \in \Hh^{\frac{1}{2}}(\Gamma_C)\:, \label{LMF:4}\\
    	\prodD{ \mathfrak{F} s, |\uu_{\tanExt}| }{\Gamma_C} - \prodD{ \muu, \uu_{\tanExt}}{\Gamma_C} &= 0\:. \label{LMF:5}
    \end{align}
\end{subequations}
Furthermore, this set of equations and inequalities can be interpreted as the first order optimality conditions associated to the following saddle point problem:
\begin{equation}
	\inf_{\vv\in\Xx} \sup_{(\eta,\xii)\in H_+\times \Bb_s} \mathcal{L}(\vv,\eta,\xii)\:,
	\label{SaddlePt}
\end{equation}
where we have used the notations:
\begin{equation*}
\begin{aligned}	
	H_+ &:=\{ \eta\in H^{-\frac{1}{2}}(\Gamma_C) \ | \ \prodD{\eta,\zeta}{\Gamma_C} \geq 0, \ \forall \zeta\geq 0 \}\:, \\
	\Bb_s &:= \{ \xii\in \Hh^{-\frac{1}{2}}(\Gamma_C) \ | \ \prodD{|\xii|-\mathfrak{F}s,\nuu}{\Gamma_C} \leq 0, \ \forall \nuu \geq 0 \}\:, \\
	\mathcal{L}(\vv,\eta,\xii) &:= \frac{1}{2}a(\vv,\vv) - L(\vv) + \prodD{\eta,\vv_{\normalExt}-\gG_{\normalExt}}{\Gamma_C} + \prodD{\xii, \vv_{\tanExt}}{\Gamma_C} \:.
\end{aligned}
\end{equation*}
Now that these five different formulations \eqref{FF:all}, \eqref{IV}, \eqref{OPT}, \eqref{LMF:all}, \eqref{SaddlePt} of the original contact problem have been introduced, we are in a position to present the two regularized formulations studied in this work.

\subsection{Two regularized formulations}

Even though the previous formulations of the contact problem are well understood from the mathematical point of view, their numerical resolution remains a challenging issue which constitutes an active research field. Indeed, due to the non-differentiabilities and the non-linearities involved in the problem, most of the classical numerical methods are either inapplicable or inefficient.
One popular way to get around this difficulty is to consider slightly different but more regular formulations, which are easier to solve. Of course, such approximate formulations are exepected to give consistent approximations of the orginal solution $\uu$. Among all methods following this approach, we consider here the two most widely used in the industry: the \textit{penalty method} and the \textit{augmented lagrangian method}.
From the point of view of optimization, the main difficulties in \eqref{IV} are the constraint $\uu\in\Kk$ and the non-differentiability of $j_T$. In what follows, we briefly recall the two regularization procedures used by the two methods mentioned above, and we give the associated regularized formulations.

\paragraph{The penalty method}

The penalty method was first introduced in the general case of constrained optimization \cite{courant1943variational}. The idea consists in replacing the constraint by adding a penalty term to the cost functional. In this way the intial constrained optimization problem is turned into an unconstrained one. Even if this approach was being used in constrained optimization since the 1940's, its first applications to the contact problem date back to the 1980's \cite{KikSon1981,WriSimTay1985,SimBer1986,kikuchi1988contact}. 
Here, the penalty approach will enable us to treat not only the constraint $\uu\in\Kk$, but also the non-smoothness of $j_T$.

Let $\varepsilon>0$ be the penalty parameter. After adding a smooth penalty term $j_\varepsilon$ to replace the constraint $\uu\in\Kk$ and introducing a regularized version $j_{T,\varepsilon}$ of $j_T$, we end up with a new minimization problem: find $\uu_\varepsilon\in\Xx$ solution to
\begin{equation}
  \underset{\vv \in \Xx}{\inf} \ \frac{1}{2}a(\vv,\vv) - L(\vv)  + j_{T,\varepsilon}(\vv) + j_\varepsilon(\vv) \: .
  \label{OPTPena}
\end{equation}
This unconstrained differentiable optimization problem can be equivalently reformulated by means of its first order optimality condition \cite{ekeland1999convex}: find $\uu_\varepsilon\in\Xx$ solution to
\begin{equation}
	\begin{aligned}
	    a(\uu_\varepsilon,\vv)  - L(\vv) +  \langle j_{T,\varepsilon}'(\uu_\varepsilon), \vv \rangle + \langle j_\varepsilon'(\uu_\varepsilon), \vv \rangle = 0, \:\:\: \forall \vv \in \Xx\: , \\
	    \Longleftrightarrow \hspace{2em} a(\uu_\varepsilon,\vv) + \frac{1}{\varepsilon} \prodL2{\maxx\left(\uu_{\varepsilon,\normalExt} -\gG_{\normalExt}\right), \vv_{\normalExt}}{\Gamma_C}  + \frac{1}{\varepsilon} \prodL2{\qq(\varepsilon\mathfrak{F}s,\uu_{\varepsilon,\tanExt}), \vv_{\tanExt}}{\Gamma_C} = L(\vv), \:\:\: \forall \vv \in \Xx\: ,
	\end{aligned}
    \label{FVPena}
\end{equation}
where we have introduced the projections $\maxx:\mathbb{R}\to\mathbb{R}$ and $\qq:\mathbb{R}_+\times\mathbb{R}^{d-1}\to\mathbb{R}^{d-1}$ defined by: $\forall t\in\mathbb{R}$ and $\forall (\alpha,z) \in \mathbb{R}_+\times\mathbb{R}^{d-1}$,
\begin{equation*}
	\maxx(t) := \max(0,t)\:, \hspace{3em}    
    \qq(\alpha,z) := \left\{
    \begin{array}{lr}
         z & \mbox{ if } |z|\leq \alpha ,\\
         \alpha\dfrac{z}{|z|} & \mbox{ else.}
    \end{array}
    \right.
\end{equation*}

\begin{remark}
	These functions are indeed projections. First, $\maxx$ is the projection onto $\mathbb{R}_+$ in $\mathbb{R}$. Then, for any fixed $\alpha\in\mathbb{R}_+$, $\qq(\alpha,\cdot)$ represents the projection onto the ball $\mathcal{B}(0,\alpha)$ in $\mathbb{R}^{d-1}$.
\end{remark}

\begin{remark}
	The explicit expressions of $j_\varepsilon$ and $j_{T,\varepsilon}$ are not given here since formulation \eqref{FVPena} only requires the knowledge of their derivatives in any direction $\vv\in\Xx$, which can be written in terms of $\maxx$ and $\qq$. The interested reader may find these expressions in the PhD thesis of the first author \cite{chaudet2019phd}.
\end{remark}

Let us finally state the existence, uniqueness and convergence result for the penalty formulation \cite{chouly2013convergence}.
\begin{theorem}
	Under our regularity asumptions on the data, there exists a unique solution $\uu_\varepsilon\in\Xx$ to \eqref{FVPena}. Moreover, $\uu_\varepsilon\to\uu$ strongly in $\Xx$ when $\varepsilon\to 0$.
\end{theorem}

\paragraph{The augmented lagrangian method}

Before presenting this method, we need to introduce the so-called \textit{augmented lagrangian formulation}. This formulation is also a regularization of the original contact problem, but contrary to the penalty formulation for which the solution $\uu_\varepsilon\neq\uu$ depends on $\varepsilon$, its solution is the same as the solution $\uu$ to the original formulation, no matter the value of the regularization parameter. The principle is to regularize the lagrangian $\mathcal{L}$ in \eqref{SaddlePt} without changing its saddle-point. This leads to a set of first order optimality conditions more regular than \eqref{LMF:all}. This approach was intially introduced at the end of the 1960's in the context of non-linear optimization under equality constraints \cite{Hes1969,Pow1969}. Then, a few decades later, it has been used to treat sliding \cite{WriSimTay1985,LanTay1986} and fricitonal \cite{AlaCur1991,SimLau1992,LauSim1993,HeeCur1993} contact problems.

Let $\gamma_1>0$ and $\gamma_2>0$ be two regularization parameters. We define the augmented lagrangian $\mathcal{L}_{\gamma}$ by: $\forall (\vv,\eta,\xii) \in \Xx\times L^2(\Gamma_C)\times\Ll^2(\Gamma_C)$,
\begin{equation}
\begin{aligned}
	\mathcal{L}_\gamma(\vv,\eta,\xii) := \frac{1}{2}a(\vv,\vv) - L(\vv) &+ \frac{1}{2\gamma_1} \left( \norml \maxx \left( \eta + \gamma_1(\vv_{\normalExt}-\gG_{\normalExt})\right) \normr_{0,\Gamma_C}^2 - \norml \eta \normr_{0,\Gamma_C}^2 \right) \\
	&+ \frac{1}{2\gamma_2} \left( \norml \qq(\mathfrak{F}s,\xii + \gamma_2\vv_{\tanExt})\normr_{0,\Gamma_C}^2 - \norml \xii \normr_{0,\Gamma_C}^2 \right) \:,
\end{aligned}
\label{DefLagAug}
\end{equation}
The saddle-point problem associated to this new lagrangian is given by:
\begin{equation}
	\inf_{\vv\in\Xx} \sup_{(\eta,\xii)\in L^2\times \Ll^2} \mathcal{L}_\gamma(\vv,\eta,\xii)\:.
	\label{SaddlePtLagAug}
\end{equation}
Given the expression of $\mathcal{L}_{\gamma}$, a necessary condition for our triplet $(\uu,\lambda,\muu)$ to coincide with the saddle-point of $\mathcal{L}_{\gamma}$ is that $(\lambda,\muu)\in L^2(\Gamma_C)\times \Ll^2(\Gamma_C)$. Since we have chosen the friction coefficient $\mathfrak{F}$ regular enough and $s\in L^2(\Gamma_C)$, one can prove that the condition $\muu\in\Ll^2(\Gamma_C)$ is always satisfied in our context \cite{stadler2004infinite}. However, in general, the regularity of $\lambda$ is only $H^{-\frac{1}{2}}(\Gamma_C)$. Therefore we need to make an additional assumption in order to be able to apply our regularization procedure \cite{stadler2004infinite,chaudet2021shape}.
\begin{assumption}
	The Lagrange multiplier $\lambda$ belongs to $L^2(\Gamma_C)$.
	\label{hyp:RegLambda}
\end{assumption}
With this asusmption, it is now possible to show using standard techniques \cite{ito2000augmented} that this augmented formulation has the desired properties. Besides, since the new saddle-point problem is posed on the vector space $\Xx\times L^2(\Gamma_C)\times\Ll^2(\Gamma_C)$, the associated first order optimality conditions take the simpler form of equalities.
\begin{theorem}
	Under assumption \ref{hyp:RegLambda}, for all $\gamma_1$, $\gamma_2>0$, problem \eqref{SaddlePtLagAug} admits a unique solution in $\Xx\times L^2(\Gamma_C)\times\Ll^2(\Gamma_C)$. Moreover, this solution coincides with $(\uu,\lambda,\muu)$ the solution to \eqref{SaddlePt}, and it can be fully characterized by its first order optimality conditions:
\begin{subequations} \label{LMFLagAug:all}
    \begin{align}
    a(\uu,\vv) - L(\vv) + \prodL2{\lambda, \vv_{\normalExt}}{\Gamma_C} + \prodL2{\muu, \vv_{\tanExt}}{\Gamma_C} = L(\vv) \:, \hspace{1em} &\forall \vv \in \Xx\:, \label{LMFLagAug:1}\\
    \lambda = \maxx\left( \lambda + \gamma_1(\uu_{\normalExt}-\gG_{\normalExt}) \right)\:, \hspace{1em} &\mbox{ a.e. on } \Gamma_C\:, \label{LMFLagAug:2}\\
    \muu = \qq(\mathfrak{F}s,\muu + \gamma_2\uu_{\tanExt}) \:, \hspace{1em} &\mbox{ a.e. on } \Gamma_C\:. \label{LMFLagAug:3}
    \end{align}
\end{subequations}
\end{theorem}
Consequently, whenever the original contact problem is regular enough (i.e. whenever Assumption \ref{hyp:RegLambda} is satisfied), we obtain an equivalent regularized formulation which rewrites as a non-linear and non-differentiable mixed variational formulation. One possible approach to solve such a formulation is to decouple the variables using a fixed-point algorithm. This is precisely the purpose of the augmented lagrangian algorithm. Let us briefly recall its different steps.

\begin{algorithm}
\caption{Augmented lagrangian method}
\begin{enumerate}
  \item Choose $(\lambda^0,\muu^0)\in L^2(\Gamma_C)\times\Ll^2(\Gamma_C)$ and initialize $k=0$.
  \item Choose $\gamma_1^{k+1}$, $\gamma_2^{k+1}>0$, then find $\uu^{k+1}\in\Xx$ the solution to, $\forall \vv \in \Xx$,
    \begin{equation} 
        a(\uu^{k+1},\vv) + \prodL2{\maxx\left( \lambda^k + \gamma_1^{k+1} (\uu^{k+1}_{\normalExt}-\gG_{\normalExt}) \right), \vv_{\normalExt}}{\Gamma_C} + \prodL2{ \qq\left( \mathfrak{F} s ,\muu^k + \gamma_2^{k+1} \uu^{k+1}_{\tanExt} )\right), \vv_{\tanExt}}{\Gamma_C} = L(\vv) \:.
        \label{IterLagAug}
    \end{equation}
        \item Update the Lagrange multipliers: 
    \begin{subequations}\label{MAJMultLagAug:all}
      \begin{align}
      \lambda^{k+1} &= \maxx\left( \lambda^k + \gamma_1^{k+1} (\uu^{k+1}_{\normalExt}-\gG_{\normalExt}) \right) \:\: \mbox{ a.e. on } \Gamma_C \: , \label{MAJMultLagAug:1}\\
        \muu^{k+1} &= \qq\left( \mathfrak{F} s ,\muu^k + \gamma_2^{k+1} \uu^{k+1}_{\tanExt} )\right)  \:\: \mbox{ a.e. on } \Gamma_C \: \label{MAJMultLagAug:2}.
      \end{align}
    \end{subequations}
    \item Update $k=k+1$ and go back to step 2, until convergence.
\end{enumerate}
\label{alg:ALM}
\end{algorithm}
As in the case of the penalty method, this method benefits from nice properties in terms of well-posedness and convergence \cite{stadler2004infinite} which are exposed in the next theorem.
\begin{theorem}
  Let $0<\gamma_1^1\leq\gamma_1^2\leq \cdots$, and $0<\gamma_2^1\leq \gamma_2^2\leq \cdots$ be the two increasing sequences of parameters. Then for any $k\geq 1$, the iterate $\uu^k\in\Xx$ exists and is unique. Moreover, the sequence $\left\{ \uu^k\right\}_k$ converges to $\uu$ strongly in $\Xx$, and the sequence of dual variables $\left\{(\lambda^k,\muu^k)\right\}_k$ converges to $(\lambda,\muu)$ weakly in $L^2(\Gamma_C)\times\Ll^2(\Gamma_C)$.
  \label{ThmCvLagAug}
\end{theorem}

\section{Shape optimization problem}

\subsection{Problem formulation}

Given a cost functional $J(\Omega)$ depending explicitly on the domain $\Omega$, and also implicitly, through $\yy(\Omega)$ the solution to the mechanical problem on $\Omega$, we search for a minimizer of $J$ among the set $\pazocal{U}_{ad}$ of admissible domains. Mathematically, the problem of minimizing $J$ with respect to $\Omega$ or \textit{shape optimization problem} reads:
\begin{equation}
    \begin{array}{cl}
         \displaystyle \inf_{\Omega \in \pazocal{U}_{ad}} J(\Omega) =& \inf \ \pazocal{J}(\Omega,\yy(\Omega))  \\
         & \mbox{subject to } \left\{
         \begin{array}{l}
              \Omega \in \pazocal{U}_{ad}\:, \\
              \yy(\Omega) \mbox{ solves the contact problem . }
         \end{array}
         \right.
    \end{array}
    \label{ShapeOPT}
\end{equation}
Of course, here, we have $\yy(\Omega)=\uu_\varepsilon(\Omega)$ or $\uu^k(\Omega)$, depending on the chosen regularized formulation. Besides, we focus on rather generic cost functionals $J(\Omega)$ of type
\begin{equation}
	J(\Omega) = \pazocal{J}(\Omega,\yy(\Omega)) := \int_\Omega j(\uu(\Omega)) \,\mathrm{dx} + \int_{\partial\Omega} k(\uu(\Omega)) \,\mathrm{ds}\:.
  \label{GeneralJType}
\end{equation}
The functions $j,k$ are chosen such that they belong to $\pazocal{C}^1(\mathbb{R}^d,\mathbb{R})$, and that their derivatives, denoted $j'$, $k'$, are Lipschitz. It is also assumed that these functions and their derivatives satisfy the usual growth conditions \cite{allaire2014shape}, for all $u$, $v \in \mathbb{R}^d$,
\begin{equation}
   \begin{aligned}
    |j(u)| \leq C\left(1+|u|^2\right)\:, \qquad |j'(u)\cdot v| \leq C |u\cdot v|\:,
     \\
    |k(u)| \leq C\left(1+|u|^2\right)\:, \qquad
    |k'(u)\cdot v| \leq C|u\cdot v|\:, 
  \end{aligned}
  \label{Condjk}
\end{equation}
for some constants $C>0$.

Now that $J$ has been specified, we may turn to the set of admissible domains. Let $D\subset \mathbb{R}^d$ be a fixed bounded smooth domain, and let $\hat{\Gamma}_D\subset\partial D$ be a part of its boundary which will be the "potential" Dirichlet boundary. This means that for any domain $\Omega\subset D$, the Dirichlet boundary associated to $\Omega$ will be defined as $\Gamma_D:=\partial\Omega \cap \hat{\Gamma}_D$. With these notations, we introduce the set $\pazocal{U}_{ad}$ of all smooth open domains $\Omega$ such that the Dirichlet boundary $\Gamma_D\subset \partial D$ is of strictly positive measure, that is:
\begin{equation}
   \pazocal{U}_{ad} := \{ \Omega \subset D \: | \: \Omega \mbox{ is smooth, and }  |\partial\Omega \cap \hat{\Gamma}_D| > 0 \}.
   \label{DefUad}
\end{equation}

\subsection{Shape sensitivity analysis}

In order to build a gradient-based shape optimization algorithm, we need to compute the derivative of $J$ with respect to the domain $\Omega$. The notion of sensitivity with respect to the shape can be defined in several ways. In this work, we follow the method of \textit{perturbation of the identity} \cite{murat1975etude,henrot2006variation}. Let us introduce $\Cc^1_b(\mathbb{R}^d) := {(\pazocal{C}^1(\mathbb{R}^d)\cap W^{1,\infty}(\mathbb{R}^d) )}^d$, equipped with the $d$-dimensional $W^{1,\infty}$ norm, denoted $\norml\cdot \normr_{1,\infty}$. In order to move the domain $\Omega$, let $\thetaa \in \Cc^1_b(\mathbb{R}^d)$ be a (small) geometric deformation vector field. The associated perturbed or transported domain in the direction $\thetaa$ will be defined as: $\Omega(t) := (\Id+t\thetaa)(\Omega)$ for any $t>0$. With these notations, we recall the definitions of the shape derivatives of $J$ and $\yy$ at $\Omega$ in the direction $\thetaa$:
$$
   dJ(\Omega)[\thetaa] := \lim_{t\searrow 0} \:\frac{1}{t}\left( J(\Omega(t)) - J(\Omega) \right) ,
	\qquad
   d\yy(\Omega)[\thetaa] :=  \lim_{t\searrow 0} \:\frac{1}{t}\left( \yy(\Omega(t)) - \yy(\Omega) \right) .
   \label{DefShaDery}
$$
Let us also recall that $J$ and $\yy$ are said to be \textit{shape differentiable} if the previous limits exist for any admissible direction $\thetaa$, and if the maps $\thetaa \mapsto dJ(\Omega)[\thetaa]$ and $\thetaa \mapsto d\yy(\Omega)[\thetaa]$ are linear continuous.
When there is no ambiguity, the shape derivatives of $J$ and $\yy$ at $\Omega$ in the direction $\thetaa$ will be simply denoted by $dJ$ and $d\yy$. Finally, we state a very classical result\cite{henrot2006variation} about shape differentiability of some specific functionals.
\begin{theorem}
	Let $\psi\in W^{1,\infty}(\mathbb{R}^d)\cap W^{2,1}(\mathbb{R}^d)$, and let $J_v$ and $J_s$ two shape functionals defined by
	$$
		J_v(\Omega) := \int_\Omega \psi(\mathrm{x})\,\mathrm{dx} \qquad \mbox{ and } \qquad J_s(\Omega) := \int_{\partial\Omega} \psi(\mathrm{x})\,\mathrm{ds} \:.
	$$
	Then $J_v$ and $J_s$ are shape differentiable, and their shape derivatives in any direction $\thetaa\in\Cc^1_b(\mathbb{R}^d)$ are given by
	$$
		dJ_v(\Omega)[\thetaa] := \int_{\partial\Omega} \psi (\thetaa\cdot\normalInt)\,\mathrm{ds} \qquad \mbox{ and } \qquad dJ_s(\Omega)[\thetaa] := \int_{\partial\Omega} \left( \partial_{\normalInt}\psi+\kappa\psi \right) (\thetaa\cdot\normalInt) \,\mathrm{ds} \:,
	$$
	where $\kappa=\divv(\normalInt)$ denotes the mean curvature of $\partial\Omega$.	
	\label{thm:ExpShapeDerVolSurf}
\end{theorem}

Next, we will present theoretical results about shape differentiability of the solutions to the two regularized formulations \eqref{FVPena} \cite{chaudet2020shape} and \eqref{LMFLagAug:all} \cite{chaudet2021shape}. In practice, the augmented lagrangian is an iterative process which will stop at some iteration $k$ when the convergence criterion is reached. Therefore, for shape sensitivity analysis, the variational formulation considered is the one verified by the converged solution, which corresponds in this case to the formulation \eqref{LMFLagAug:all} at the last iteration $k$. It has already been mentioned that the variational formulations we are dealing with are not classically differentiable, therefore performing shape sensitivity analysis in this context can not be done using standard arguments. However, one can prove that these formulations are directionally differentiable. The approach followed in the two references cited above is based on this weaker notion of differentiablity. In order to better understand and identify the technical difficulties we are facing, let us introduce the sets of points where non-differentiablities occur for each formulation. 

\begin{figure}[h]
\begin{center}
    \begin{tabular}{|c|c|c|c|}
    \hline
    $\:$ & Original & Penalty & Augmented lagrangian \\
    \hline
    Weak contact set & $\{ \uu_{\normalExt}-\gG_{\normalExt}=0 \mbox{ and } \sigmaa_{\normalInt\normalExt}(\uu)=0\}$ & $\{\uu_{\varepsilon,\normalExt}-\gG_{\normalExt}=0\}$ & $\{\lambda^{k-1}+\gamma_1^k(\uu^k_{\normalExt}-\gG_{\normalExt})=0\}$  \\
    \hline
    Weak sticking set & $\{ \uu_{\tanExt}=0 \mbox{ and } |\sigmaa_{\normalInt\tanExt}(\uu)|=\mathfrak{F}s\}$ & $\{|\uu_{\varepsilon,\tanExt}|=\varepsilon\mathfrak{F}s\}$ & $\{|\muu^{k-1}+\gamma_2^k\uu^k_{\tanExt}|=\mathfrak{F}s\}$ \\
    \hline
    \end{tabular}
    \caption{Subsets of $\Gamma_C$ where non-differentiabilities occur.}
    \label{fig:TabSubNonDiff}
\end{center}
\end{figure}
In the table in Figure \ref{fig:TabSubNonDiff}, we have defined the \textit{weak contact set} and the \textit{weak sticking set} for the three different configurations \eqref{IV}, \eqref{FVPena} and \eqref{LMFLagAug:all} following a denomination that is often used \cite{beremlijski2014shape}. From the mechanical point of view, these sets represent zones where changes of state occur. For instance, for the original formulation, if a point $\mathrm{x}$ belongs to the weak contact set, then it is in contact ($\uu_{\normalExt}-\gG_{\normalExt}=0$) but there is no contact pressure ($\sigmaa_{\normalInt\normalExt}(\uu)=0$). Similarly, if $\mathrm{x}$ belongs to the weak sticking set, then is at the same time in sliding contact ($|\sigmaa_{\normalInt\tanExt}(\uu)|=\mathfrak{F}s$) and in sticking contact ($\uu_{\tanExt}=0$). From the mathematical point of view, these sets gather all points where the non-differentiabilities of $\maxx$ and $\qq$ are reached.

\begin{assumption}\label{A2}
    The weak contact set and the weak sticking set associated to \eqref{FVPena} and \eqref{LMF:all} are of surface measure 0.
\end{assumption}

Making this assumption enables us to prove shape differentiability\cite{chaudet2020shape,chaudet2021shape,chaudet2019phd} of $\uu_\varepsilon$ and $\uu^k$. Then, shape differentiability of $J$ follows from standard arguments\cite{henrot2006variation}, and we may apply Céa's method \cite{cea1986conception} to obtain an explicit expression of $dJ$ in each case by introducing adequate adjoint states. Before stating this result, we introduce the functions $G_+:\mathbb{R}\to\mathcal{L}(\mathbb{R})$ and $G_s:\mathbb{R}^{d-1}\to\mathcal{L}(\mathbb{R}^{d-1})$ such that: for all $t\in\mathbb{R}$, for all $z\in\mathbb{R}^{d-1}$,
$$
	G_+(t):= \left\{
    \begin{array}{lr}
         0 & \mbox{ if } t\leq 0 ,\\
         1 & \mbox{ else,}
    \end{array}
    \right.
    \ \ \mbox{ and } \ \
    G_s(z) := \left\{
    \begin{array}{lr}
         \Id & \mbox{ if } |z|\leq \mathfrak{F}s ,\\
         \frac{\mathfrak{F}s}{|z|}\big(\Id  -  \frac{1}{|z|^2} z \otimes  z\big) & \mbox{ else,}
    \end{array}
    \right.
$$
where $z\otimes z$ denotes the exterior product between $z$ and itself, also written $z z^T$ in the matrix representation.
These functions, which are generalized derivatives \cite{stadler2004semismooth} of $\maxx$ and $\qq$, will be useful to write the adjoint variational formulations. Another quantity of interest is the linear form $L_{adj}$, which depends only on $J$ and $\yy$:
\begin{equation}
	L_{adj}[\yy](\vv)= -\int_{\Omega} j'(\yy)\cdot\vv \,\mathrm{dx} -\int_{\partial\Omega}k'(\yy)\cdot\vv\,\mathrm{ds}  \:, \ \ \forall \vv \in \Xx\:.
	\label{RhsAdj}
\end{equation}
With these notations, we can define the adjoint states $\pp_\varepsilon$ and $\pp^k$ in $\Xx$ such that:
\begin{subequations} \label{FVAdj:all}
    \begin{align}
		a(\pp_\varepsilon,\vv)+\frac{1}{\varepsilon}\prodL2{G_+(\uu_{\varepsilon,\normalExt}-\gG_{\normalExt})\pp_{\varepsilon,\normalExt},\vv_{\normalExt}}{\Gamma_C} +\frac{1}{\varepsilon}\prodL2{G_{\varepsilon s}(\uu_{\varepsilon,\tanExt})\pp_{\varepsilon,\tanExt},\vv_{\tanExt}}{\Gamma_C} = L_{adj}[\uu_\varepsilon](\vv)\:, \ \ \forall \vv \in \Xx\:, \label{FVAdj:1}\\
		a(\pp^k,\vv)+\gamma_1^k\prodL2{G_+(\lambda^{k-1}+\gamma_1^k(\uu^k_{\normalExt}-\gG_{\normalExt}))\pp^k_{\normalExt},\vv_{\normalExt}}{\Gamma_C}+\gamma_2^k\prodL2{G_s(\muu^{k-1}+\gamma_2^k\uu^k_{\tanExt})\pp^k_{\tanExt},\vv_{\tanExt}}{\Gamma_C}= L_{adj}[\uu^k](\vv)\:, \ \ \forall \vv \in \Xx\:. \label{FVAdj:2} 
	\end{align}
\end{subequations}	

\begin{lemma}
	For any $\varepsilon>0$ and any $k\geq 1$, the solutions $\pp_\varepsilon$ and $\pp^k$ to the two formulations in \eqref{FVAdj:all} exist and are unique in $\Xx$.
	\label{lem:ExistUniqAdj}
\end{lemma}

\begin{proof}
	Since we are dealing with linear variational formulations posed on the Hilbert space $\Xx$, we can apply Lax-Milgram lemma. Let us start with the penalty adjoint formulation. First, it is clear that $L_{adj}[\uu_\varepsilon]$ is linear and continuous on $\Xx$ due to the growth assumptions \eqref{Condjk}. Besides the left hand-side is a bilinear form on $\Xx\times\Xx$. Then, using the expressions of $G_+$ and $G_s$, Cauchy-Schwarz inequality and the trace theorem, one gets that for all $(\ww,\vv)\in\Xx\times\Xx$,
$$
	\left| \prodL2{G_+(\uu_{\varepsilon,\normalExt}-\gG_{\normalExt})\ww_{\normalExt},\vv_{\normalExt}}{\Gamma_C} \right| \leq  \prodL2{|\ww_{\normalExt}|, |\vv_{\normalExt}|}{\Gamma_C} \leq C\norml\ww \normr_{\Xx}\norml\vv\normr_{\Xx} \:,
$$	
$$
	\left| \prodL2{G_{\varepsilon s}(\uu_{\varepsilon,\tanExt}-\gG_{\normalExt})\ww_{\tanExt},\vv_{\tanExt}}{\Gamma_C} \right| \leq \prodL2{|\ww_{\tanExt}|,|\vv_{\tanExt}|}{\Gamma_C} \leq C\norml\ww\normr_{\Xx} \norml\vv\normr_{\Xx} \:.
$$
Finally, we notice that for all $t\in\mathbb{R}$, $z\in\mathbb{R}^{d-1}$, the functions $G_+(t)$ and $G_s(z)$ verify:
$$
	G_+(t)y\cdot y \geq 0\:, \ \ \forall y\in\mathbb{R}, \qquad \mbox{ and } \qquad G_s(z)h\cdot h \geq 0\:, \ \ \forall h\in\mathbb{R}^{d-1}\:.
$$
Therefore, it follows from the coercivity of $a$ that the left hand-side of the formulation is also a coercive bilinear form. Lax-Milgram lemma enables us to conclude that $\pp_\varepsilon\in\Xx$ exists and is unique. The proof for $\pp^k$ can be done using exactly the same arguments.
\end{proof}

\begin{theorem} \label{thm:ExpDJ}
	Under Assumption \ref{A2}, the solutions $\uu_\varepsilon$ and $\uu^k$ to \eqref{FVPena} and \eqref{LMFLagAug:all} are shape differentiable. Thus the cost functional $J$ is also shape differentiable in these cases. Moreover, when $\uu_\varepsilon$, $\uu^k$, $\pp_\varepsilon$, $\pp^k\in \Hh^2(\Omega)\cap\Xx$, the shape derivative of $J$ in any direction $\thetaa \in \Cc^1_b(\mathbb{R}^d)$ is given by
	$$
		dJ(\Omega)[\thetaa] = \int_{\partial\Omega} \mathfrak{g}(\thetaa\cdot\normalInt) \,\mathrm{ds}\:.
	$$
	with $\mathfrak{g} = \mathfrak{g}_\varepsilon$ or $\mathfrak{g}^k$ depending on the formulation considered. More specifically, one has
	\begin{equation*}
	\begin{aligned}
		\mathfrak{g}_\varepsilon &= j(\uu_\varepsilon)+\Aa:\epsilonn(\uu_\varepsilon):\epsilonn(\pp_\varepsilon)-\ff\pp_\varepsilon + \chi_{\Gamma_N}(\kappa+\partial_{\normalInt})\left( k(\uu_\varepsilon) - \tauu\pp_\varepsilon\right)\\
		\: & \hspace{1em} + \frac{1}{\varepsilon}\chi_{\Gamma_C}(\kappa+\partial_{\normalInt})\left( \maxx(\uu_{\varepsilon,\normalExt}-\gG_{\normalExt})\pp_{\varepsilon,\normalExt} + \qq(\varepsilon\mathfrak{F}s,\uu_{\varepsilon,\tanExt})\pp_{\varepsilon,\tanExt}\right) \:, \\
		\mathfrak{g}^k &= j(\uu^k)+\Aa:\epsilonn(\uu^k):\epsilonn(\pp^k)-\ff\pp^k + \chi_{\Gamma_N}(\kappa+\partial_{\normalInt})\left( k(\uu^k) - \tauu\pp^k\right)\\
		\: & \hspace{1em} + \chi_{\Gamma_C}(\kappa+\partial_{\normalInt})\left( \lambda^k\pp^k_{\normalExt} + \muu^k\pp^k_{\tanExt}\right) \:, 
	\end{aligned}
	\end{equation*}
	where $\chi_{\Gamma_N}$ and $\chi_{\Gamma_C}$ represent the characteristic functions of $\Gamma_N$ and $\Gamma_C$.
\end{theorem}

\begin{proof}
	The result about shape differentiablity of $\uu_\varepsilon$ and $\uu^k$ has been obtained in previous works \cite{chaudet2020shape,chaudet2021shape}. The reader is referred to these references for the detailed proof of this part of the theorem. 
	
	Now, the shape differentiability property for these solutions enables us to apply Céa's method to get an expression of $dJ$. As in the proof of Lemma \ref{lem:ExistUniqAdj}, since both formulations are very similar in terms of technical difficulties, we focus on the (formal) proof in the case of the penalty method, and leave the case of the augmented lagrangian method to the reader. Let us first define the Lagrangian $\pazocal{L}_\varepsilon$ by, for all $(\Omega,\vv,\ww)\in \pazocal{U}_{ad}\times\Xx\times\Xx$,
\begin{equation*}
	\begin{aligned}
		\pazocal{L}_\varepsilon(\Omega,\vv,\ww)  :&= a(\vv,\ww) + \frac{1}{\varepsilon} \prodL2{\maxx\left(\vv_{\normalExt} -\gG_{\normalExt}\right), \ww_{\normalExt}}{\Gamma_C}  + \frac{1}{\varepsilon} \prodL2{\qq(\varepsilon\mathfrak{F}s,\vv_{\tanExt}), \ww_{\tanExt}}{\Gamma_C} - L(\ww) + \pazocal{J}(\Omega,\vv) \\
		\: & = \int_{\Omega} \Aa : \epsilonn(\vv) :\epsilonn(\ww) \, \mathrm{dx} + \frac{1}{\varepsilon} \int_{\Gamma_C}\maxx\left(\vv_{\normalExt} -\gG_{\normalExt}\right) \ww_{\normalExt} \mathrm{ds} + \frac{1}{\varepsilon} \int_{\Gamma_C}\qq(\varepsilon\mathfrak{F}s, \vv_{\tanExt}) \ww_{\tanExt} \mathrm{ds} \\
		\: & \qquad - \int_\Omega \ff \ww \mathrm{dx} - \int_{\Gamma_N} \tauu \ww \mathrm{ds} + \int_\Omega j(\vv)\, \mathrm{dx} + \int_{\partial\Omega} k(\vv) \, \mathrm{ds} \:.
	\end{aligned}
\end{equation*}
Note that, given this definition, the following equality holds for all $\ww\in\Xx$,
$$
	J(\Omega) = \pazocal{J}(\Omega,\uu_\varepsilon(\Omega)) = \pazocal{L}_\varepsilon(\Omega,\uu_\varepsilon(\Omega),\ww)\:.
$$
Differentiating this relation with respect to the shape gives us an expression of the shape derivative of $J$, valid for any $\ww\in\Xx$:
\begin{equation*}
	\begin{aligned}
		dJ(\Omega)[\thetaa] &= \left\langle \frac{d}{d\Omega} \left(\pazocal{L}_\varepsilon(\Omega,\uu_\varepsilon(\Omega),\ww)\right); \thetaa \right\rangle  \\
		\: & = \left\langle \frac{\partial\pazocal{L}_\varepsilon}{\partial\Omega}(\Omega,\uu_\varepsilon,\ww) ; \thetaa \right\rangle + \left\langle \frac{\partial\pazocal{L}_\varepsilon}{\partial\vv}(\Omega,\uu_\varepsilon,\ww) ; d\uu_\varepsilon \right\rangle \:.
	\end{aligned}
\end{equation*}
It is possible to cancel the second term in the last expression by taking $\ww=\pp_\varepsilon$. Indeed, with this specific choice, this term is equal to the variational formulation \eqref{FVAdj:all} with test-function $\vv=d\uu_\varepsilon$. Therefore one finally gets
$$
	dJ(\Omega)[\thetaa] = \left\langle \frac{\partial\pazocal{L}_\varepsilon}{\partial\Omega}(\Omega,\uu_\varepsilon,\pp_\varepsilon) ; \thetaa \right\rangle \:,
$$
which leads to the desired formula after application of Theorem \ref{thm:ExpShapeDerVolSurf}.
\end{proof}

\begin{remark}
	For example, Assumption \ref{A2} is satisfied when all weak contact points and all weak sticking points represent a finite number of points in 2D or a finite number of curves in 3D.
\end{remark}

\begin{remark}
	The fact that $dJ$ depends only on the normal component $\thetaa\cdot\normalInt$ of $\thetaa$ on the boundary $\partial\Omega$ is not specific to our context. It is actually a consequence of the Hadamard-Zolésio structure theorem \cite{sokolowski1992introduction}.
\end{remark}

Now that we have obtained expressions of the shape derivatives of $J$ for our two regularized problems which are rather easy to compute in practice, we can build our gradient-based shape optimization algorithm.

\section{Description of the algorithm}

Let us first recall that we aim at solving the shape optimization problem:
\begin{equation*}
    \inf_{\Omega \in \pazocal{U}_{ad}} J(\Omega) \:,
\end{equation*}
where $J$ is a functional of type \eqref{GeneralJType}, $\Omega$ is the domain on which the contact problem is posed, and $\pazocal{U}_{ad}$ is the set of admissible shapes, as defined in \eqref{DefUad}. In order to solve this problem numerically, we propose a gradient-based optimization algortihm which generates a sequence of shapes  $\{ \Omega^l\}_l$ such that the monotonicity of $J$ is ensured. In other words, the generated sequence of shapes is such that $J(\Omega^{l+1})<J(\Omega^l)$ for each $l$. Here is a brief presentation of the different steps of the algorithm.
\begin{algorithm}
\caption*{\textbf{Algorithm:} Shape optimization}
\begin{enumerate}
	\item Choose a domain $\Omega^0$ and intialize $l=0$. 
	\item Solve the primal formulation \eqref{FVPena} or \eqref{LMFLagAug:all} in $\Omega^l$.
	\item Solve the adjoint formulation \eqref{FVAdj:1} or \eqref{FVAdj:2} in $\Omega^l$.
	\item Find a descent direction $\thetaa^l$.
	\item Update the domain $\Omega^l \to \Omega^{l+1}$.
    \item Update $l=l+1$ and go back to step 2, until convergence is reached.
\end{enumerate}
\caption{Shape otpimization algortihm}
\label{alg:ShapeOpt}
\end{algorithm}

\begin{remark}
	In practice, at the beginning of the iterative process, it might be useful to accept shapes $\Omega^{l+1}$ such that $J(\Omega^{l+1})< \beta J(\Omega^l)$, with $\beta>1$. This prevents the algorithm from falling too fast in the neighbourhood of a local minimum.
\end{remark}

The goal of this section is to describe the body of the algorithm, namely steps 2, 3, 4 and 5. In these descriptions, we will ommit the superscript $l$ related to the current iteration for the reader's convenience. For example, the variables $\Omega^l$, $\thetaa^l$ will be replaced by $\Omega$, $\thetaa$. Besides, since the domains considered are reprensented by means of a level-set function, we begin with a short and very formal presentation of the level-set method in the context of shape optimization.

\subsection{Level-set reprensentation of shapes}

The level-set method was first introduced by Osher and Sethian \cite{OshSet1988} to give an implicit representation of the boundary of a smooth domain $\Omega\subset \mathbb{R}^d$. More specifically, $\Omega$ is associated to the set of negative values of an auxiliary function $\phi$ defined on the whole space $\mathbb{R}^d$ (in practice, $\phi$ is defined on a computational domain $D$ assumed to be large enough). Similarly, the boundary $\partial\Omega$ is defined as the zero level-set of $\phi$. In other words, $\phi$ satisfies:
\begin{equation}
    \left\{ \
    \begin{array}{ll}
         \phi(\mathrm{x}) < 0 & \mbox{ if } \mathrm{x} \in \Omega\:,  \\
         \phi(\mathrm{x}) = 0 & \mbox{ if } \mathrm{x} \in \partial\Omega \:, \\
         \phi(\mathrm{x}) > 0 & \mbox{ if } \mathrm{x} \in D\setminus \overline{\Omega}\:.
    \end{array}
    \right.
    \label{DefLS}
\end{equation}
A typical choice of function $\phi$ is the signed distance function \cite{DelZol2001} to $\partial\Omega$.
One of the main advantages of this method is that it enables us to represent a domain $\Omega(t)$ which evolves during a time interval $[0,T]$ under the action of a velocity field $\thetaa$. Indeed, it suffices to consider a function $\phi$ depending on both space and time, i.e. for all $t\in[0,T]$, $\phi(t,\cdot)$ is the level-set function associated to $\Omega(t)$. Let us assume that we start with a domain $\Omega(0)=\Omega^0$, represented by some function $\phi^0$. Then the evolution of $\Omega$ with respect to $t$ can be modeled by the following partial differential equation on $[0,T]\times \mathbb{R}^d$ (referred to as the \textit{level-set advection equation}), associated with a suitable initial condition,
\begin{equation}
    \begin{aligned}
    &\frac{\partial\phi}{\partial t}(t,\mathrm{x}) + \thetaa(t,\mathrm{x})\cdot \grad\phi(t,\mathrm{x}) = 0\:, \\
    &\phi(0,\mathrm{x}) = \phi^0(\mathrm{x})\:.
    \label{LSAdvEquation}
    \end{aligned}
\end{equation}
Moreover, if we choose a velocity field $\thetaa$ directed along the direction of the normal $\normalInt$ (as suggested by the result of Theorem \ref{thm:ExpDJ}), say $\thetaa=\theta \normalInt$, then \eqref{LSAdvEquation} can be rewritten as:
\begin{equation}
    \begin{aligned}
    &\frac{\partial\phi}{\partial t}(t,\mathrm{x}) + \theta(t,\mathrm{x}) \cdot|\grad\phi(t,\mathrm{x})| = 0\:, \\
    &\phi(0,\mathrm{x}) = \phi^0(\mathrm{x})\:,
    \label{HJEquation}
    \end{aligned}
\end{equation}
which is now a \textit{Hamilton-Jacobi equation}. 

The nice properties of this method have made it very popular in the context of shape optimization since its first developments in the beginning of the 2000's \cite{allaire2002level,allaire2004structural}. Compared to approaches based on an explicit and fully discrete representation of the shape, the level-set approach is an interesting alternative as it gives a continuous and smooth (albeit implicit) representation of the shape all along the geometric deformation process. Furthermore, it is also naturally well suited to deal with changes of topology \cite{allaire2005structural}.

\subsection{Resolution of the primal formulation}

Let $D_h$ be the discretization of the domain $D$, where $h$ stands for the mesh size. As it will be seen in details in subsection \ref{subsec:DomEv}, at each iteration of Algorithm \ref{alg:ShapeOpt}, our cutting/remeshing procedure provides us with a discretized version $\Omega_h$ of $\Omega$ as a subdomain of $D_h$. This subdomain $\Omega_h$ is our computational domain for the resolution of the contact problem. In order to maximize the efficiency of this cutting/remeshing procedure, we consider unstructured meshes made of triangles in 2D and tetrahedra in 3D. This enables us to use standard Lagrange finite elements (typically $P^1$ or $P^2$) to discretize our variables.

\paragraph{Case of the penalty method.}
We choose a discretization $\uu_{\varepsilon,h}\in P^2$ of $\uu_\varepsilon$, for which the convergence analysis has been established \cite{chouly2013convergence}. Since the continuous problem \eqref{FVPena} is non-linear and non-differentiable, we solve it using a semi-smooth Newton method \cite{stadler2004semismooth}.

\paragraph{Case of the augmented lagrangian method.}
In this case, we choose a discretization $\uu^{k}_h\in P^2$ of $\uu^{k}$ and $\left( \lambda^{k}_h, \muu^{k}_h\right)\in P^1\times P^1$ of $\left( \lambda^{k}, \muu^k\right)$. The convergence analysis has already been performed \cite{burman2019augmented} for such a discretization of \eqref{LMFLagAug:all}. Note that the augmented lagrangian method can also be understood as a special case of Nitsche's method, for which the convergence analysis has also been established \cite{chouly2013nitsche,chouly2014adaptation}. As in the case of the penalty method, we use a semi-smooth Newton approach.

\subsection{Resolution of the adjoint formulation}

Given that the adjoint state lies in the same Sobolev space as the primal state, we use the same discretization for both of them, namely $P^2$. Then, we recall that unlike the primal formulation, the adjoint formulation is linear, which makes it easier to solve. The right hand-side is given by \eqref{RhsAdj} and can be easily assembled using the primal state $\uu_{\varepsilon,h}$ or $\uu^k_h$ that has just been computed.
As far as the left hand-side is concerned, let us first notice that the adjoint formulation \eqref{FVAdj:all} is obtained by differentiating the primal formulation with respect to the state variable. Consequently, in order to build the algebraic system solved by $\pp_{\varepsilon,h}$ (or $\pp^k_h$), we may reuse the matrix from the last Newton iteration that led to $\uu_{\varepsilon,h}$ (or $\uu^k_h$).

\begin{remark}
	In practice, computing the adjoint state is not very costly as it is similar to one Newton iteration in terms of complexity. Here, it is even less costly since the assembly and the $LU$ factorization of the matrix have already been performed.
\end{remark}

\subsection{Computation of a descent direction}

Given that the evolution of the shape is simulated by solving \eqref{HJEquation}, we need the vector field $\thetaa=\theta \normalInt$ to be defined on the whole computational domain $D$. This vector field is also expected to be a descent direction, i.e. it must satisfy $dJ(\Omega)[\thetaa]<0$. A rather common choice \cite{de2006velocity,allaire2004structural} consists in taking $\theta\in H^1(D)$ solution to:
\begin{equation}
	\int_{D} \gradd \theta\cdot\gradd v + \alpha\,\theta\, v \: \mathrm{dx} = -dJ(\Omega)[v\normalInt] = - \int_{\partial\Omega} \mathfrak{g} \,v \, \mathrm{ds}\:, \ \ \forall v \in H^1(D)\:,
	\label{FVThetaScal}
\end{equation}
where the parameter $\alpha>0$ is usually of the same order of magnitude as the mesh size $h$ \cite{dapogny2013shape}, and $\mathfrak{g}=\mathfrak{g}_\varepsilon$ or $\mathfrak{g}^k$ depending on the formulation considered (cf Theorem \ref{thm:ExpDJ}). Obviously, such a choice ensures that $\thetaa$ is a descent direction. Besides, formulation \eqref{FVThetaScal} can be discretized using standard Lagrange finite elements, e.g. $\theta_h\in P^1$ or $P^2$ in $D_h$.

\subsection{Domain evolution}
\label{subsec:DomEv}

In order to find the new (perturbed) domain given a vector field $\thetaa=\theta\normalInt$ computed at the previous step, we aim at solving \eqref{HJEquation} on a time interval $[0,T]$ taking $\theta(t,\mathrm{x})=\theta(\mathrm{x})$ in $[0,T]\times D$. In this work, this equation is solved using finite differences (in time and space) on an auxiliary cartesian grid. There are two main motivations for this choice. The first one is essentially practical: the numerical resolution of \eqref{HJEquation} on a cartesian grid using the fast-marching method \cite{harten1987uniformly,osher1991high} is efficient, easy to implement and extremely robust, which is an essential feature when solving a shape optimization problem. Second, having two different meshes for the level-set equation and the mechanical problem enable us to work with a very fine cartesian grid, which gives us a very accurate and smooth representation of the shape, whithout increasing the computational cost related to solving the contact problem. This can be interpreted as a two-level approach: the representation of the shape (low computational cost) is considered at a fine level and the mechanical quantities as well as the shape gradients (high computational cost) are considered at a coarse level. 

In practice, we take $\phi$ as the signed distance function to $\partial\Omega$. Then, given a grid $D_\Delta$ of $D$ consisting of the nodes $(x_i,y_j,z_k)$, we get a discrete $\phi_{ijk}^n = \phi\left(t^n,(x_i,y_j,z_k)\right)$ on $D_\Delta$ for each $n\in\{0,\dots, N\}$, where $t^0=0$ and $t^N=T$. After evaluating the descent direction $\theta_h$ defined in $D_h$ at each grid point of $D_\Delta$, we end up with $\theta_{ijk}$ defined on $D_\Delta$. Now, taking $\phi^0_{ijk}$ as the signed distance computed at the last step $t^N$ of the previous shape optimization iteration, we solve \eqref{HJEquation} and obtain $\phi^N_{ijk}$, the representation of the new domain after evolution. Finally, projecting this result onto the finite element mesh by means of a $L^2$-projection yields an implicit representation of the new shape on $D_h$, denoted by $\phi_h\in P^m$, $m\geq 1$.

\begin{remark}
	Even if $\phi^0$ is very smooth and represents a distance function, there is no guarantee that these nice properties still hold at each step of the numerical resolution of \eqref{HJEquation}. In practice, the solution may become irregular, especially in the neighbourhood of $\{ \phi=0 \}$. It is also possible that at some point, one gets $|\grad \phi| \ll 1$ ou $\gg 1$, which would mean that $\phi$ is no longer a good approximation of a distance function. In order to avoid these issues, one may periodically \textit{reinitialize} during the resolution of \eqref{HJEquation}. More specifically, for some fixed $\bar{t}\in [0,T]$, reinitializing $\bar{\phi}=\phi(\bar{t},\cdot)$ consists in replacing $\bar{\phi}$ by the solution $\psi$ to:
	\begin{equation}
		\left\{ \
		\begin{array}{cr}
			\partial_t \psi + \mbox{sgn} (\bar{\phi}) \left( |\grad\psi|-1 \right) = 0  & \mbox{ in } [0,+\infty)\times D,\\
			\psi(0,x) = \bar{\phi}(\mathrm{x}) & \mbox{ in } D.	
		\end{array}
		\right.
		\label{EqRedistPhi}
	\end{equation}
	Since \eqref{EqRedistPhi} is a Hamilton-Jacobi equation, we may use the same numerical scheme as before to solve it.
\end{remark}

\paragraph{Building $\Omega_h$.} 
At this stage, we would like to decide whether the new shape is accepted or rejected. Thus we should be able to check if it satisfies the geometric criteria necessary to belong to $\pazocal{U}_{ad}$, and we should be able to compute $J(\Omega_h)$. From the point of view of a conformal discretization approach, this means that the discrete domain $\Omega_h$ has to be constructed. Since we know $\phi_h$ on $D_h$, we can compute the intersection between the coarse level boundary $\{ \phi_h=0\}$ and the edges of $D_h$. Then, it suffices to add the nodes corresponding to these intersections to $D_h$, as well as all other necessary mesh components (edges, faces, elements) necessary to ensure the validity of the new mesh. We end up with a new mesh $\tilde{D}_h$ of $D$ that contains a submesh $\Omega_h$ representing the new domain.

This approach simply consists in \textit{cutting} the mesh around the coarse level-set $\{ \phi_h=0\}$, and it is not new \cite{lorensen1987marching} in the case of a field $\phi_h\in P^1$. It has also been used in the context of shape optimization \cite{allaire2011topology,allaire2013mesh,allaire2014shape} as the preliminary step of a more sophisticated remeshing procedure. This method has two main advantages: it is easy to implement, and robust. Nevertheless, it may generate meshes with low quality elements (strechted, potentially very small) and the interface $\partial\Omega_h$ might be irregular, especially in 3D.

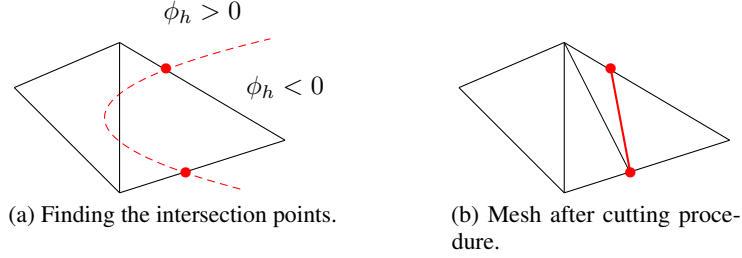
\begin{figure}[h]
\begin{center}

\subfloat[Finding the intersection points.]{
\begin{tikzpicture}

\draw[scale=0.5,domain=-1.9:2.1,smooth,variable=\y,red, densely dashed]  plot ({\y*\y},{\y});     

\node[] at (1.3,1.4) {$\phi_h>0$};
\node[] at (2.4,0.4) {$\phi_h<0$};

\draw[black] (0.2,-1) -- (0.2,1);
\draw[black] (0.2,-1) -- (2.4,-0.3);
\draw[black] (0.2,1) -- (2.4,-0.3);
\draw[black] (0.2,-1) -- (-1.2,0.4);
\draw[black] (0.2,1) -- (-1.2,0.4);

\node[red] at (0.82,0.64) {$\bullet$};
\node[red] at (1.08,-0.74) {$\bullet$};

\end{tikzpicture}
}
\hspace{4em}
\subfloat[Mesh after cutting procedure.]{
\begin{tikzpicture}
    
\draw[black] (0.2,-1) -- (0.2,1);
\draw[black] (0.2,-1) -- (2.4,-0.3);
\draw[black] (0.2,1) -- (2.4,-0.3);
\draw[black] (0.2,-1) -- (-1.2,0.4);
\draw[black] (0.2,1) -- (-1.2,0.4);

\draw[red, thick] (1.08,-0.74) -- (0.82,0.64);
\draw[black, ultra thin] (1.08,-0.74) -- (0.2,1);

\node[red] at (0.82,0.64) {$\bullet$};
\node[red] at (1.08,-0.74) {$\bullet$};

\end{tikzpicture}
}

\end{center}
  \caption{Mesh-cutting procedure for $\phi_h\in P^1$.}
  \label{fig:SchDecP1}
\end{figure}

\begin{figure}[h]
\begin{center}

\subfloat[Finding the intersection points.]{
\begin{tikzpicture}

\draw[scale=0.5,domain=-1.9:2.1,smooth,variable=\y,red, densely dashed]  plot ({\y*\y},{\y});     

\node[] at (1.3,1.4) {$\phi_h>0$};
\node[] at (2.4,0.4) {$\phi_h<0$};

\draw[black] (0.2,-1) -- (0.2,1);
\draw[black] (0.2,-1) -- (2.4,-0.3);
\draw[black] (0.2,1) -- (2.4,-0.3);
\draw[black] (0.2,-1) -- (-1.2,0.4);
\draw[black] (0.2,1) -- (-1.2,0.4);

\node[red] at (0.2,0.3) {$\bullet$};
\node[red] at (0.2,-0.32) {$\bullet$};
\node[red] at (0.82,0.64) {$\bullet$};
\node[red] at (1.08,-0.74) {$\bullet$};

\end{tikzpicture}
}
\hspace{4em}
\subfloat[Mesh after cutting procedure.]{
\begin{tikzpicture}
    
\draw[black] (0.2,-1) -- (0.2,1);
\draw[black] (0.2,-1) -- (2.4,-0.3);
\draw[black] (0.2,1) -- (2.4,-0.3);
\draw[black] (0.2,-1) -- (-1.2,0.4);
\draw[black] (0.2,1) -- (-1.2,0.4);

\draw[red, thick] (0.2,0.3) -- (0.2,-0.32);
\draw[red, thick] (0.2,0.3) -- (0.82,0.64);
\draw[red, thick] (0.2,-0.32) -- (1.08,-0.74);
\draw[black, ultra thin] (0.2,0.3) -- (-1.2,0.4);
\draw[black, ultra thin] (0.2,0.3) -- (2.4,-0.3);
\draw[black, ultra thin] (0.2,-0.32) -- (-1.2,0.4);
\draw[black, ultra thin] (0.2,-0.32) -- (2.4,-0.3);

\node[red] at (0.2,0.3) {$\bullet$};
\node[red] at (0.2,-0.32) {$\bullet$};
\node[red] at (0.82,0.64) {$\bullet$};
\node[red] at (1.08,-0.74) {$\bullet$};

\end{tikzpicture}
}

\end{center}
  \caption{Mesh-cutting procedure for $\phi_h\in P^2$.}
  \label{fig:SchDecP2}
\end{figure}
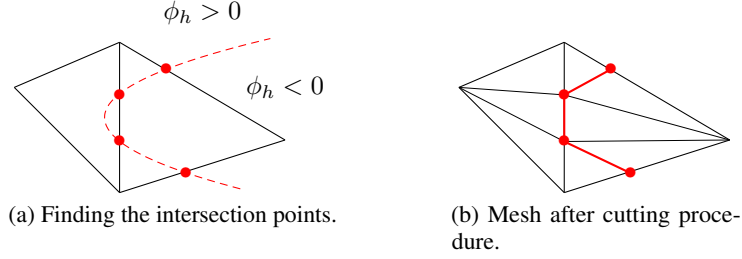

In this work, we choose to follow this approach despite its drawbacks, but we propose to increase its accuracy by considering a field $\phi_h\in P^m$ with $m>1$. This enables us to find a greater number of intersection points between the level set $\{ \phi_h=0\}$ and $D_h$, and it leads to a more accurate representation of the underlying smooth boundary $\{ \phi=0\}$. As illustrated in Figure \ref{fig:SchDecP1} and Figure \ref{fig:SchDecP2}, using a $P^2$ interpolation instead of a $P^1$ interpolation for $\phi_h$ really seems to make a difference in terms of accuracy, even in rather simple 2D cases.

\paragraph{Example}
\label{ex:decoupage}

	Let us consider a three-dimensional example in order to illustrate the influence of the interpolation degree of $\phi_h$ on $\Omega_h$. The discrete computational domain $D_h$ consists in a regular mesh of the cube $D=[0,1]^3$ consisting in 4913 vertices, see Figure \ref{sub:MailDh}. Let us define the sclar polynomial function of degree 4  
$$
	\phi(\mathrm{x}):= 16\left(x-\frac{1}{2}\right)^4+\left(y-\frac{1}{2}\right)^2+\left(z-\frac{1}{2}\right)^2-\frac{1}{4}\:.
$$
In Figure \ref{fig:ResDecSphere}, we have displayed the result after performing our mesh-cutting procedure in the three different cases $\phi_h\in P^1$, $P^2$ and $P^3$.
\begin{figure}[h]
  \begin{center}
    \subfloat[Domain $\Omega_h$ for $\phi_h\in P^1$.]{
    \includegraphics[width=0.28\textwidth]{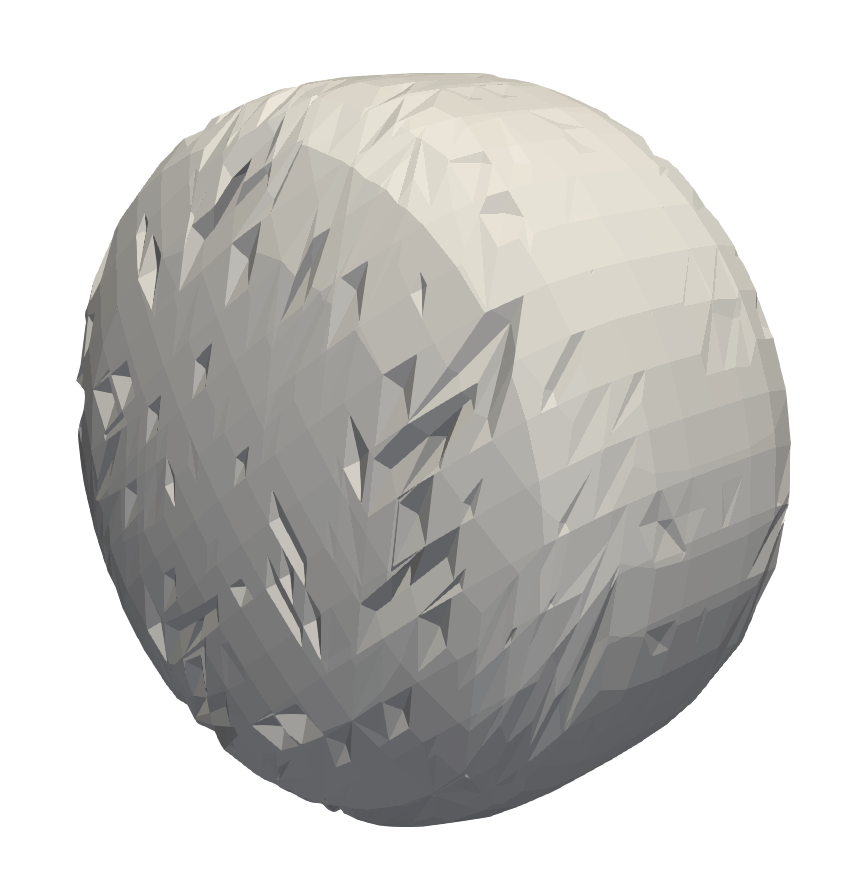}
    \label{sub:OmegahP1}
    }
    \hspace{1em}
    \subfloat[Domain $\Omega_h$ for $\phi_h\in P^2$.]{
    \includegraphics[width=0.28\textwidth]{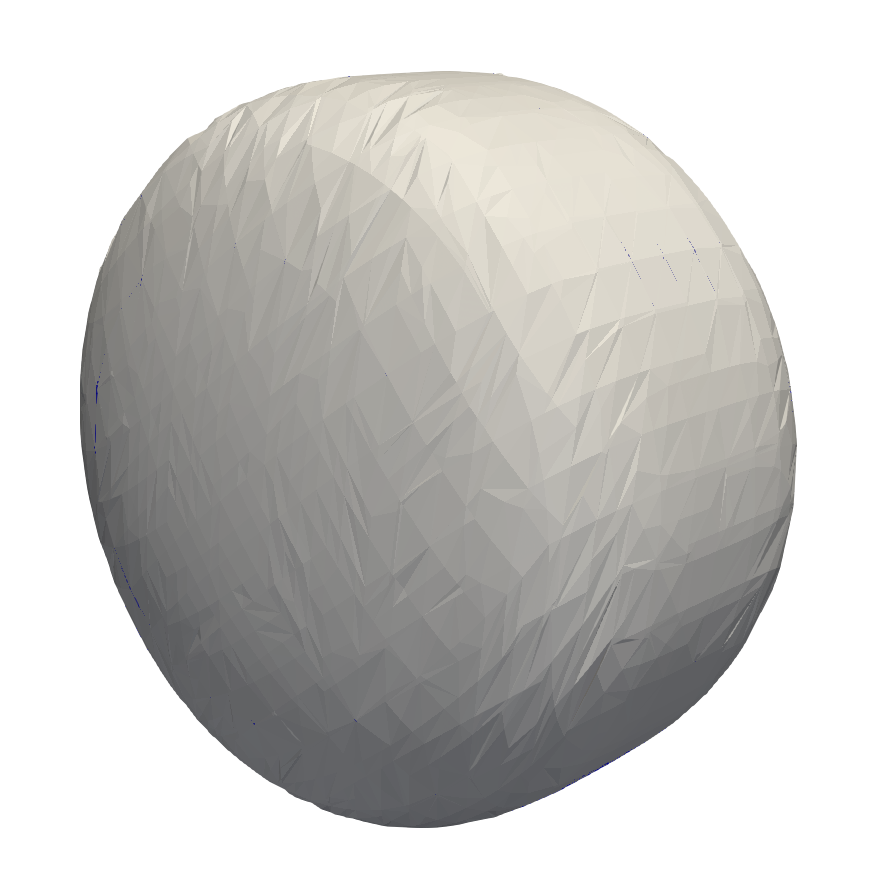}
    \label{sub:OmegahP2}
    }
    \hspace{1em}
    \subfloat[Domain $\Omega_h$ for $\phi_h\in P^3$.]{
    \includegraphics[width=0.28\textwidth]{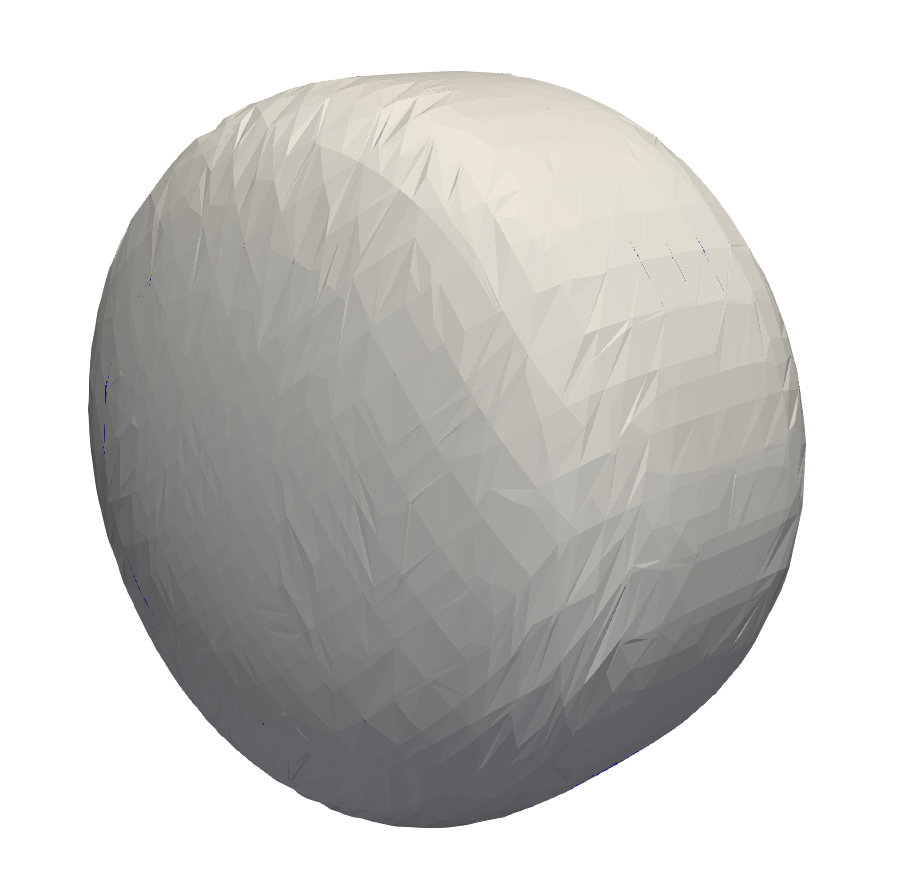}
    \label{sub:OmegahP3}
    }
    \end{center}
  \caption{Domains $\Omega_h$ after cutting the mesh around the surface $\{ \phi_h=0 \}$.}
  \label{fig:ResDecSphere}
\end{figure}
As expected, the surfaces $\{ \phi_h=0 \}$ obtained when using $P^2$ and $P^3$ interpolations are smoother, for a similar number of additional vertices: 3563 in the $P^1$ case, 3574 in the $P^2$ case and 3579 in the $P^3$ case. However, there is no significant improvement when using a cubic interpolation compared to a quadratic one.

In practice, the observations made in this example remain true in most cases. Therefore, we choose to work with $\phi_h\in P^2$, which in our context appears to be the best compromise. Besides, as it can be seen in Figure \ref{fig:MailDecSphere}, the choice of interpolation degree does not seem to influence the quality of the resulting mesh.
	
\begin{figure}[h]
  \begin{center}
    \subfloat[Mesh $D_h$.]{
    \includegraphics[width=0.23\textwidth]{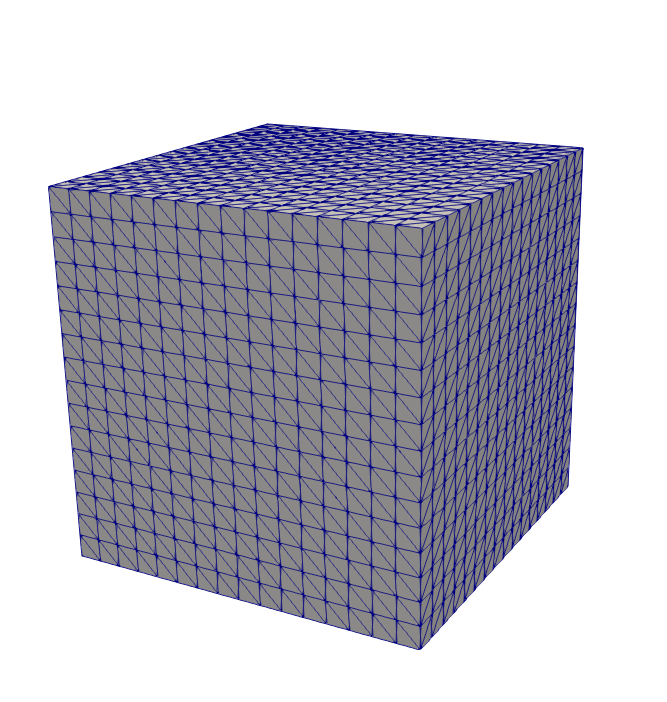}
    \label{sub:MailDh}
    }
    \subfloat[Mesh for $\phi_h\in P^1$.]{
    \includegraphics[width=0.23\textwidth]{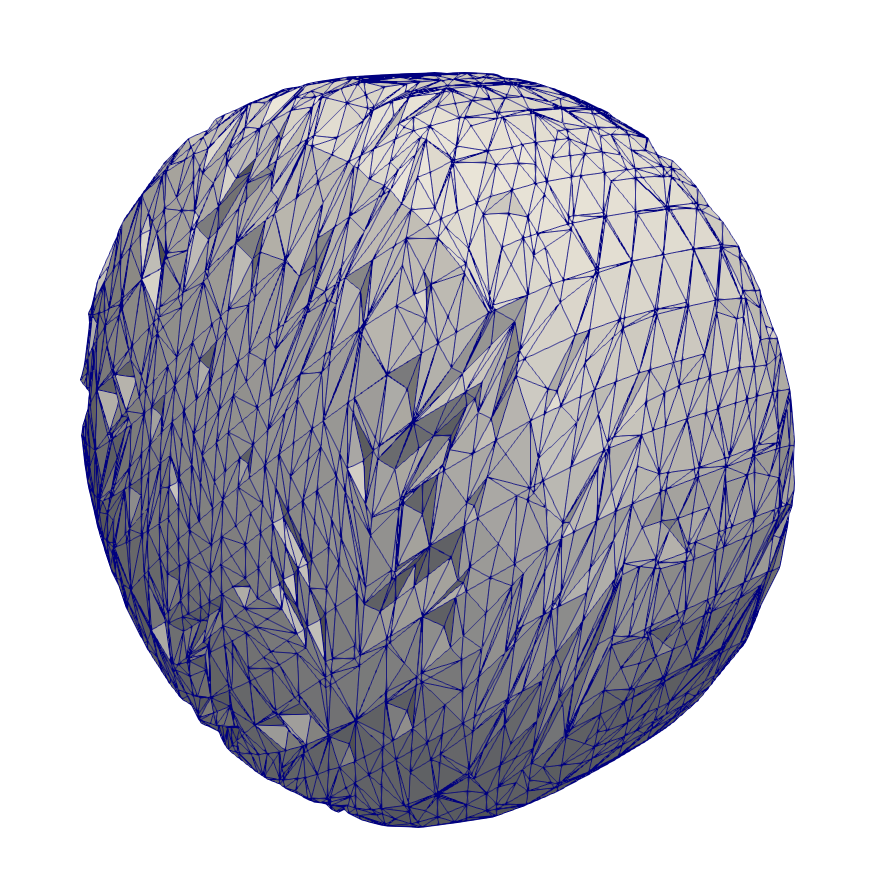}
    \label{sub:MailOmegahP1}
    }
    \subfloat[Mesh for $\phi_h\in P^2$.]{    \includegraphics[width=0.23\textwidth]{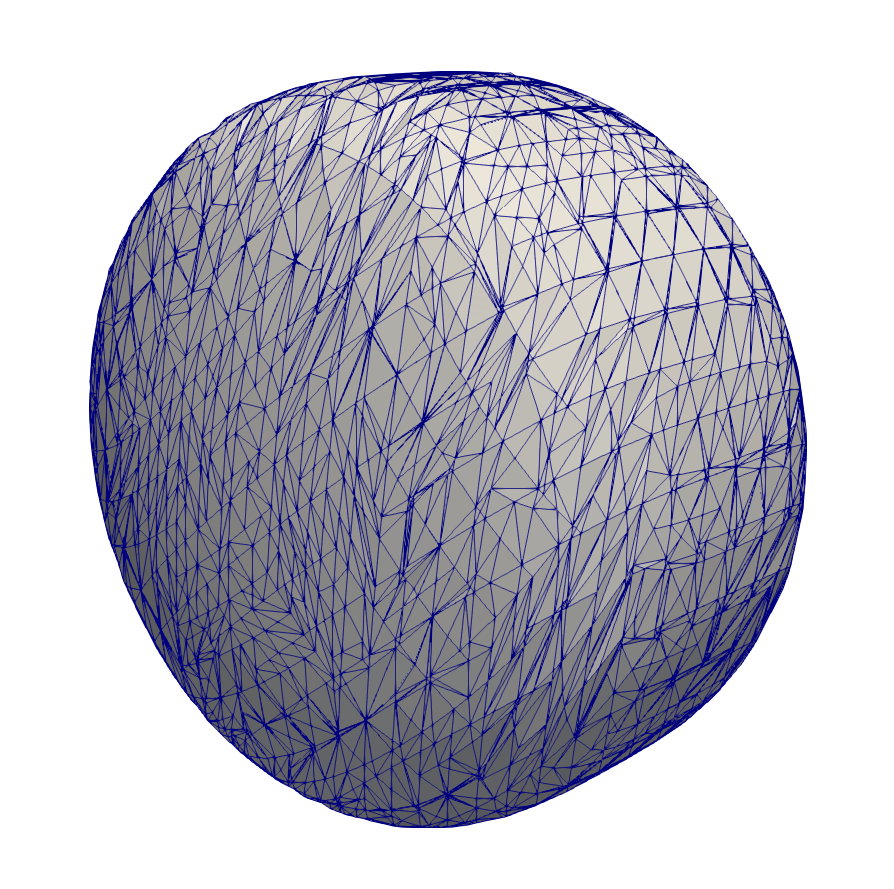}
    \label{sub:MailOmegahP2}
    }
    \subfloat[Mesh for $\phi_h\in P^3$.]{    \includegraphics[width=0.23\textwidth]{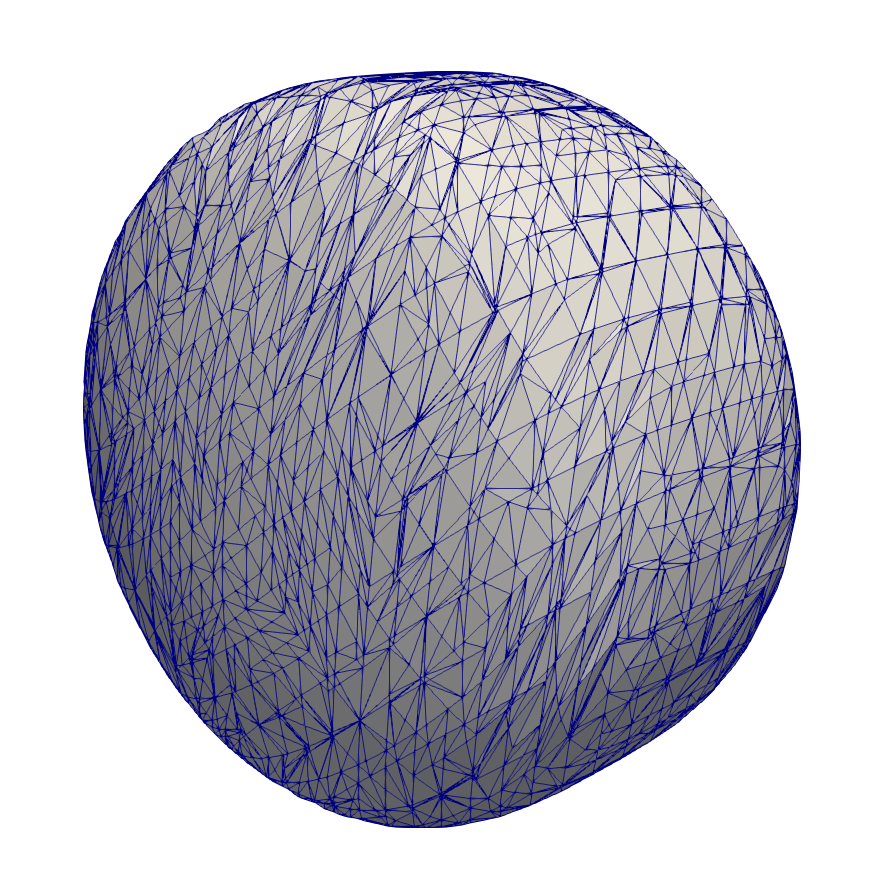}
    \label{sub:MailOmegahP3}
    }
    \end{center}
  \caption{Meshes after cutting procedure around the surface $\{ \phi_h=0 \}$.}
  \label{fig:MailDecSphere}
\end{figure}

\begin{remark}
	In the previous example, we see that the resulting surface mesh of $\{ \phi_h=0 \}$ contains very small and stretched elements. This could be an obstacle for the numerical resolution when the problem considered is particularly difficult to solve (large number of unknowns, strong non-linearities, etc). In this work, we restrict ourselves to the case of linearized elasticity, on meshes containing at most 5000 vertices. Therefore we do not encounter such difficulties when solving the mechanical problem on these meshes. Moreover, even if something wrong happened at some iteration, the shape optimization process would not be jeopardized thanks to the robustness of the optimization solver. Nevertheless, in order to extend our shape optimization method to industrial benchmarks (fine meshes, complex geometries) in the case of contact with Coulomb friction and non-linear elastic materials, it would be necessary to add a step that improves the quality of the mesh. An idea would be to control the quality of the mesh, then perform a complete or partial mesh adaptation procedure \cite{dapogny2013shape}.
\end{remark}

\section{Numerical results}

In this section, we present numerical results obtained when applying our shape optimization algortihm to some benchmarks in two and three dimensions. More specifically, we get interested in two types of configurations. First, in order to validate our approach, we apply our algorithm in a specific configuration coming from the literature \cite{stromberg2010topology,maury2017shape} where the contact zone is known a priori. In this configuration, we fix a set of zones $\hat{\Gamma}_C\subset \partial D$ where contact phenomena may occur. Then, during the optimization process, the potential contact zone $\Gamma_C$ associated to a given shape $\Omega$ will be defined as $\Gamma_C:=\hat{\Gamma}_C\cap \partial\Omega$. This means that the treatment of $\Gamma_C$ is similar to the one of $\Gamma_D$: if the shape "holds on" to a part of $\hat{\Gamma}_C$, then there will be contact in this zone, otherwise the boundary is free of constraint. Especially, no shape derivative needs to be computed on $\Gamma_C$ in this case. Second, we consider configurations where the contact zone is not known a priori. We only define a rigid body $\Omega_{rig}$ such that the deformable body might come in contact with $\Omega_{rig}$ under the action of external forces and surface loads. Then, given a shape $\Omega$, the computation of the gap will enable us to determine which points come in contact with $\Omega_{rig}$ among the set of all points $\partial\Omega \setminus (\Gamma_D\cup \Gamma_N)$. In this more general configuration, we may use the shape derivative on $\Gamma_C$ to also optimize this part of the boundary without enforcing any a priori constraint. To the best of our knowledge, there does not exist any three-dimensional benchmark in the literature for such configurations in the case of non planar rigid bodies. Therefore we propose a new one which is inspired by a two-dimensional benchmark \cite{fancello1995numerical}. 

The algorithm has been implemented in MEF++, the finite element based multiphysics platform developed at the GIREF (\textit{Groupe Interdisciplinaire de Recherche en Éléments Finis}, at Laval University). The interested reader is referred to the website \url{https://giref.ulaval.ca/} for further information about this research code.

In what follows, we always consider isotropic linear elastic materials with Young modulus $E=1$ and Poisson coefficient $\nu=0.3$. Besides, as it is often the case in structural optimization, the cost functional $J$ is defined as a linear combination of the volume $Vol$ and the compliance $C$:
$$
    J(\Omega) := \alpha_1 C(\Omega) + \alpha_2 Vol(\Omega)= \int_\Omega (\alpha_1\ff \yy(\Omega)+\alpha_2)\, \mathrm{dx}+ \int_{\Gamma_N} \alpha_1\tauu \yy(\Omega) \,\mathrm{ds}\:,
$$
where $\alpha_1$ and $\alpha_2$ are real positive coefficients and $\yy=\uu_\varepsilon$ or $\uu^k$ depending on the formulation solved. 

\begin{remark}
	As it has been mentioned before, the level-set method is well-suited to treat changes of topology, especially to close or merge "holes" (i.e. connected components of $D\setminus\Omega$). Therefore, a common practice \cite{allaire2004structural,allaire2005structural} consists in taking an initial shape $\Omega^0$ containing holes. This can be seen as a way to enlarge the set of admissible shapes used in practice.
\end{remark}

\subsection{Example where the contact zone is known a priori}

\paragraph{Two-dimensional bridge.} 
This first benchmark comes from the literature \cite{maury2016shape}. In this reference, the authors solve a classically differentiable contact formulation obtained using the penalty method with an additional regularization step. Their representation of the shapes is also based on the level-set method, but they do not have a mesh $\Omega_h$. Instead, the boundary conditions are imposed weakly on $\partial\Omega\setminus D$. Due to the similarities between their method and ours,  it is interesting to tackle this benchmark and compare our result to theirs in order to validate our algorithm. The computational domain is given by $D=[0,1]\times[0,1]$. Two disjoint potential contact zones are placed at the bottom-left and bottom-right parts of $\partial D$, and a Neumann boundary condition is enforced at the bottom-center of $\partial D$, see Figure \ref{sub:InitCLAnchrage2d}. As for physical forces, we set $\ff=0$ and $\tauu=(0,-0.01)$. The coefficients in the expresion of $J$ satisfy $\alpha_1=25$, $\alpha_2=0.01$. In the case of the penalty method, the penalty parameter is such that $\varepsilon = 10^8$, and in the case of the augmented lagrangian method, we take $\gamma_1^k=\gamma_2^k=1000$ for all $k$. Finally, in the frictional case, the Tresca threshold is set to $s=10^{-2}$ (order of magnitude of the normal constraint $\sigmaa_{\normalInt\!\normalExt}$), and the friction coefficient is $\mathfrak{F}=0.2$. The optimal designs obtained in each case are displayed in Figure \ref{fig:ResAnchrage2d}.

\begin{figure}[h]
  \begin{center}
    \subfloat[Initial geometry.]{
	\resizebox{0.29\textwidth}{!}{
    \begin{tikzpicture}
	\draw[black] (0,0) -- (0.5,0);
	\draw[dartmouthgreen, very thick] (0.25,0) -- (0.75,0);
	\draw[black] (0.75,0) -- (1.75,0);
	\draw[orange, very thick] (1.75,0) -- (2.25,0);
	\draw[black] (2.25,0) -- (3.25,0);
	\draw[dartmouthgreen, very thick] (3.25,0) -- (3.75,0);
	\draw[black] (3.75,0) -- (4,0);
	\draw[black] (4,0) -- (4,4);
	\draw[black] (4,4) -- (0,4);
	\draw[black] (0,4) -- (0,0);

	\node[white] at (2,-0.25) {$\:$};
	\node[white] at (4.2,2) {$\:$};
	\node[black] at (2,2) {$D$};
	\node[orange] at (2,0.4) {$\Gamma_N$};
	\node[dartmouthgreen] at (0.5,0.4) {$\Gamma_C$};
	\node[dartmouthgreen] at (3.5,0.4) {$\Gamma_C$};
	\end{tikzpicture}
    \label{sub:InitCLAnchrage2d}
    }
    }
   \hspace{.3em}
    \subfloat[Iteration 0.]{
    \includegraphics[width=0.3\textwidth]{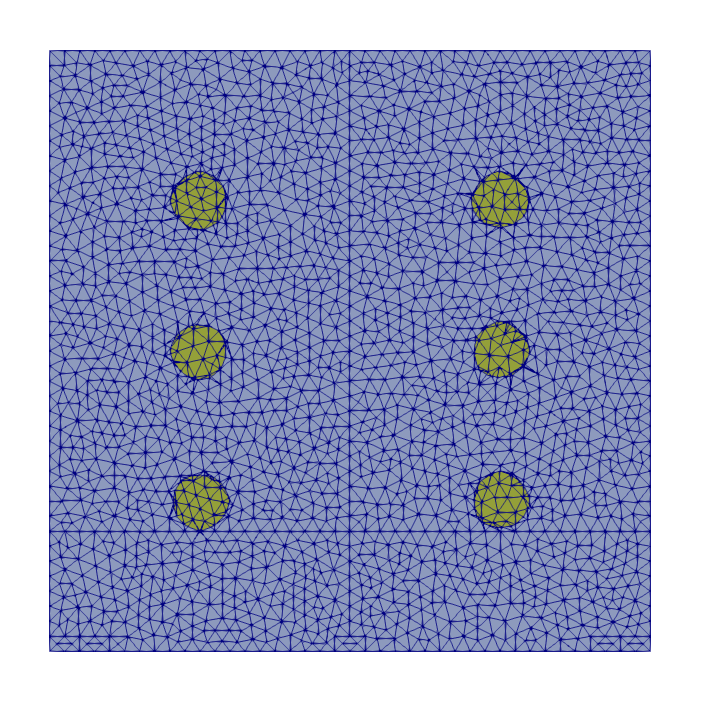}
    \label{sub:ResAnchrage2dInit}
    }
   \hspace{.3em}
    \subfloat[Final iteration for sliding contact (penalty).]{
    \includegraphics[width=0.3\textwidth]{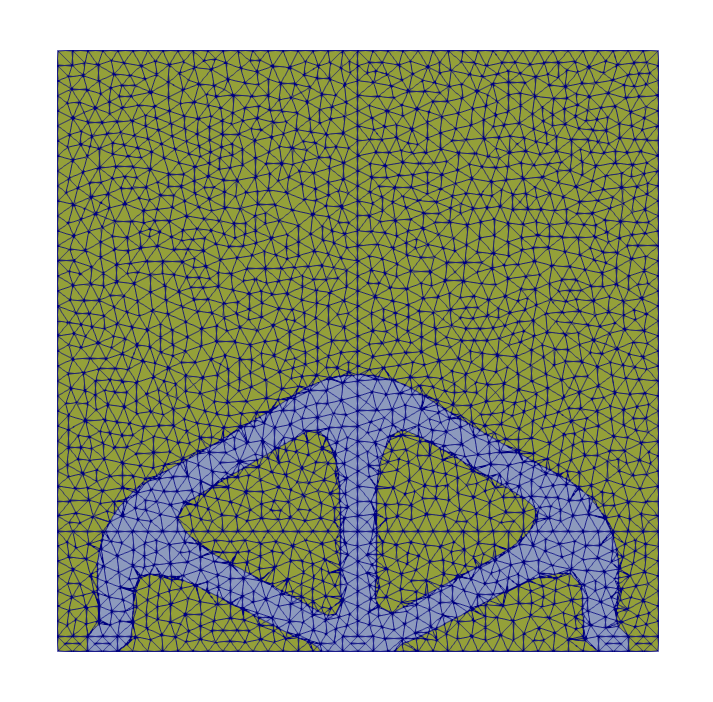}
    \label{sub:ResAnchrage2dPena}
    }
    
    \subfloat[Final iteration for sliding contact (augmented lagrangian).]{
    \includegraphics[width=0.3\textwidth]{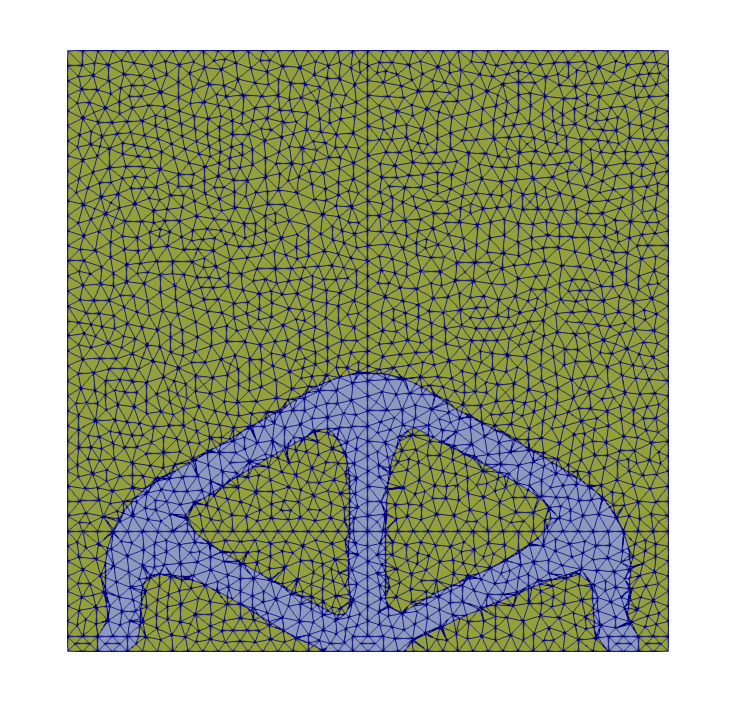}
    \label{sub:ResAnchrage2dLagAug}
    }
    \hspace{.3em}
    \subfloat[Final iteration for frictional contact (penalty).]{
    \includegraphics[width=0.29\textwidth]{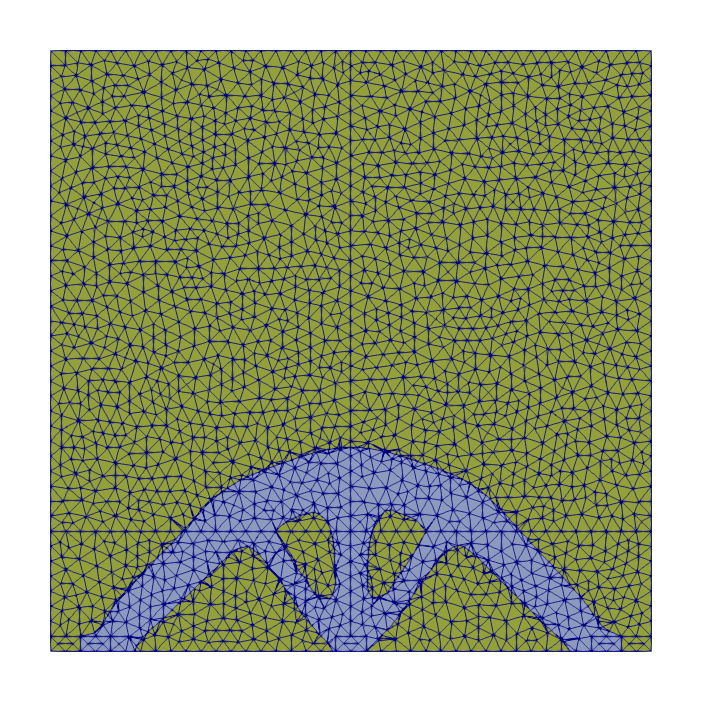}
    \label{sub:ResAnchrage2dFrottPena}
    }
    \hspace{.3em}
    \subfloat[Final iteration for frictional contact (augmented lagrangian).]{
    \includegraphics[width=0.29\textwidth]{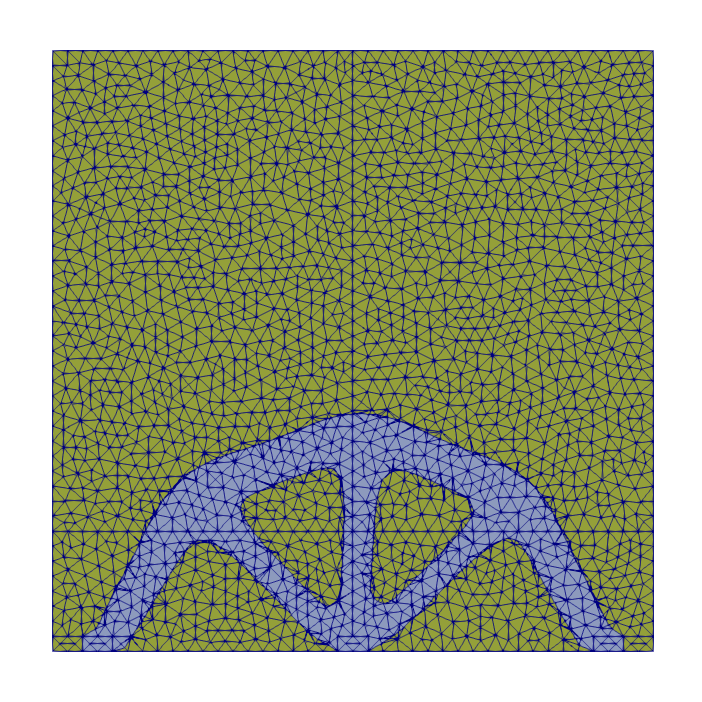}
    \label{sub:ResAnchrage2dFrottLagAug}
     }
    \end{center}
  \caption{Initial and final designs for the 2D bridge  ($\Omega_h$ in blue, $D_h\setminus \Omega_h$ in yellow).}
  \label{fig:ResAnchrage2d}
\end{figure}
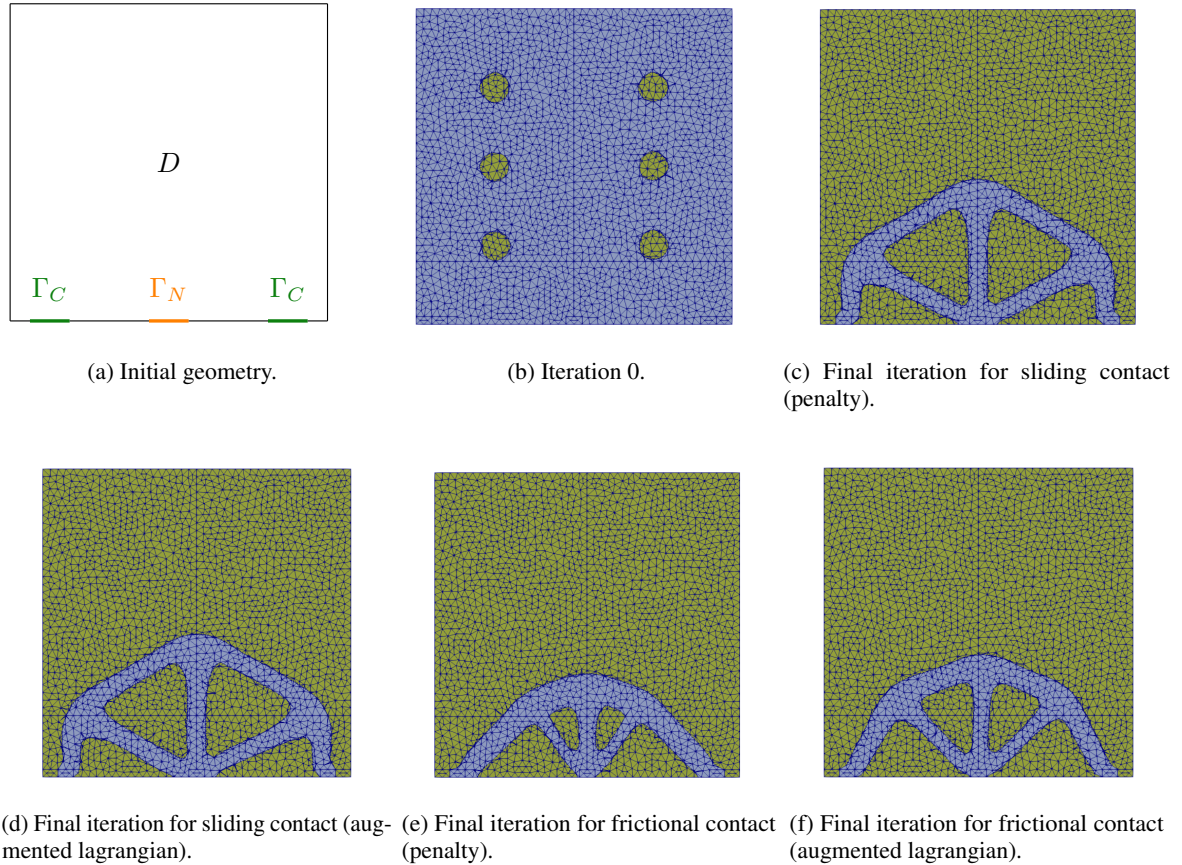

In the case of sliding contact, the optimal designs obtained for both formulation (penalty and augmented lagrangian) are very similar. They are also very close to the solution obtained in the work mentioned above \cite{maury2017shape}. Especially, we note that the legs of the bridge are vertical above the contact zone, this limitates sliding and improves the rigidity of the structure. The convergence curves of $J(\Omega^l)$, plotted in Figure \ref{fig:CvgIterAnchrage2d}, also have similar behaviours for the two formumations, and they seem to converge to the same value. However, in the frictional case, the optimal designs look quite different for the two formulations, but the values of $J$ are very close, see Figure \ref{fig:CvgIterAnchrage2d}. We still obtain shapes with properties similar to the one in the reference. As expected, thanks to the friction, the legs of the bridge may incline without deteriorating the rigidity of the structure.

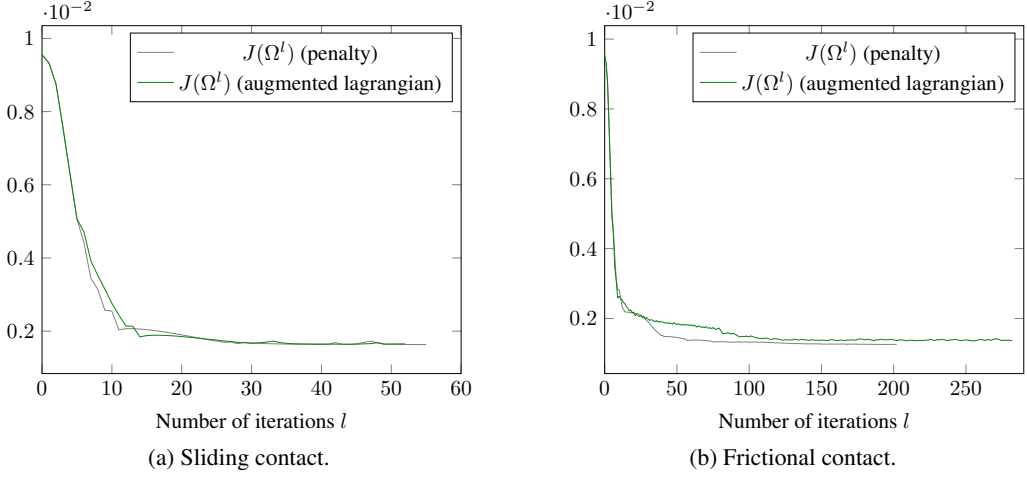
\begin{figure}
\begin{center}
    \subfloat[Sliding contact.]{   
	\resizebox{!}{0.35\textwidth}{
	\begin{tikzpicture}
		\begin{axis}[
    		xlabel={Number of iterations $l$},
    		xmin=0, xmax=60,
    		]
    		\addplot[color=gray,mark=none] table {res_anchrage2d_pena.txt};
    		\addplot[color=dartmouthgreen,mark=none] table {res_anchrage2d_lagaug.txt};
   			\legend{$J(\Omega^l)$ (penalty), $J(\Omega^l)$ (augmented lagrangian)}
   		\end{axis}
	\end{tikzpicture}
	}
	}
	\hspace{2em}
    \subfloat[Frictional contact.]{   
	\resizebox{!}{0.35\textwidth}{
	\begin{tikzpicture}
		\begin{axis}[
    		xlabel={Number of iterations $l$},
    		xmin=0, xmax=290,
    		]
    		\addplot[color=gray,mark=none] table {res_anchrage2d_frott_pena_bis.txt};
    		\addplot[color=dartmouthgreen,mark=none] table {res_anchrage2d_frott_lagaug.txt};
   			\legend{$J(\Omega^l)$ (penalty), $J(\Omega^l)$ (augmented lagrangian)}
   		\end{axis}
	\end{tikzpicture}
	}
	}
\end{center}
\caption{Convergence curves for the 2D bridge.}
\label{fig:CvgIterAnchrage2d}
\end{figure}

\subsection{Examples where the contact zone is not known a priori}

In order to study the validity of the results from Theorem \ref{thm:ExpDJ} in numerical practice, we consider a case which tests the shape derivatives on the contact zone.

\paragraph{Two-dimensional cantilever in contact with a disk.}

Inspired by classical shape optimization benchmarks in linear elasticity  (without contact) and by benchmarks from contact problems resolution, we propose to study the case of a 2D cantilever in contact with a disk \cite{stadler2004infinite}. This case has already been considered in shape optimization \cite{fancello1995numerical}, as well as other closely related variants where the rigid body is flat \cite{haslinger1987shape,haslinger2003introduction}. We prefer to consider the more general case of a disk because in the case of a plane, the normal $\normalExt$ does not depend on $\thetaa$ and therefore it has no influence on the shape derivative. In the works cited above, only the contact zone is optimized and this zone is represented as the graph of a finite element function on a given mesh (paradigm \textit{discretize-then-optimize}). Here, following the paradigm \textit{optimize-then-discretize}, we account for optimization of the whole boundary of $\Omega$ (including $\Gamma_C$) and allow changes of topology thanks to the level-set representation.

We consider the usual cantilever benchmark on the rectangle $D=[0,2]\times[0,1]$, where the Neumann zone $\Gamma_N$ and the potential Dirichlet zone $\hat{\Gamma}_D$ are defined as in Figure \ref{fig:InitCLCanti2dDisque}. The only difference with the well-known case is that we add a rigid disk of radius $R$ and center $(1,-R)$ below $D$. When constraints $\ff$ and $\tauu$ are applied, the structure might come in contact with the disk.
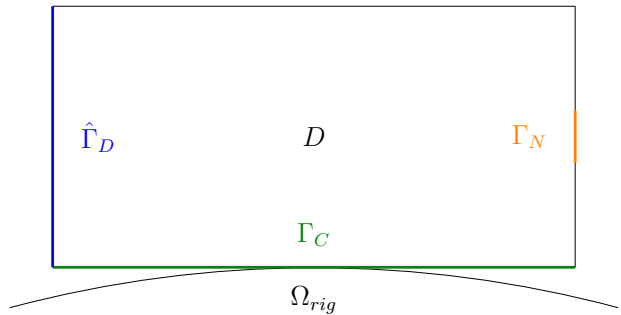
\begin{figure}[h]
	\begin{center}
	\resizebox{0.5\textwidth}{!}{
    \begin{tikzpicture}
    
	\draw[black] (8.66,-0.62) arc (75:105:18);
	\draw[dartmouthgreen, very thick] (0,0) -- (8,0);
	\draw[black] (8,0) -- (8,1.6);
	\draw[orange, very thick] (8,1.6) -- (8,2.4);
	\draw[black] (8,2.4) -- (8,4);
	\draw[black] (8,4) -- (0,4);
	\draw[blue, very thick] (0,4) -- (0,0);

	\node[white] at (4,-0.3) {$\:$};
	\node[black] at (4,2) {\large{$D$}};
	\node[blue] at (0.7,2) {\large{$\hat{\Gamma}_D$}};
	\node[orange] at (7.3,2) {\large{$\Gamma_N$}};
	\node[dartmouthgreen] at (4,0.5) {\large{$\Gamma_C$}};
	\node[black] at (4,-0.5) {\large{$\Omega_{rig}$}};
	\end{tikzpicture}
    }	
	\end{center}
  	\caption{Initial geometry for the 2D cantilever in contact with a disk.}
    \label{fig:InitCLCanti2dDisque}
\end{figure}
In this numerical example, we take $\ff=0$ et $\tauu=(0,-0.01)$. The coefficients of the cost functional are set to $\alpha_1=15$ and $\alpha_2=0.01$. We take $\varepsilon=10^5$ for the penalty formulation and $\gamma_1^k=\gamma_2^k=100$, for all $k$, for the augmented lagrangian formulation. The radius of the disk $R$ is set to 8. Finally, in the frictional case, we take as in the previous benchmark $s=10^{-2}$ and $\mathfrak{F}=0.2$. With this dataset, we obtain the results displayed in Figure \ref{fig:ResCanti2dCercle}.

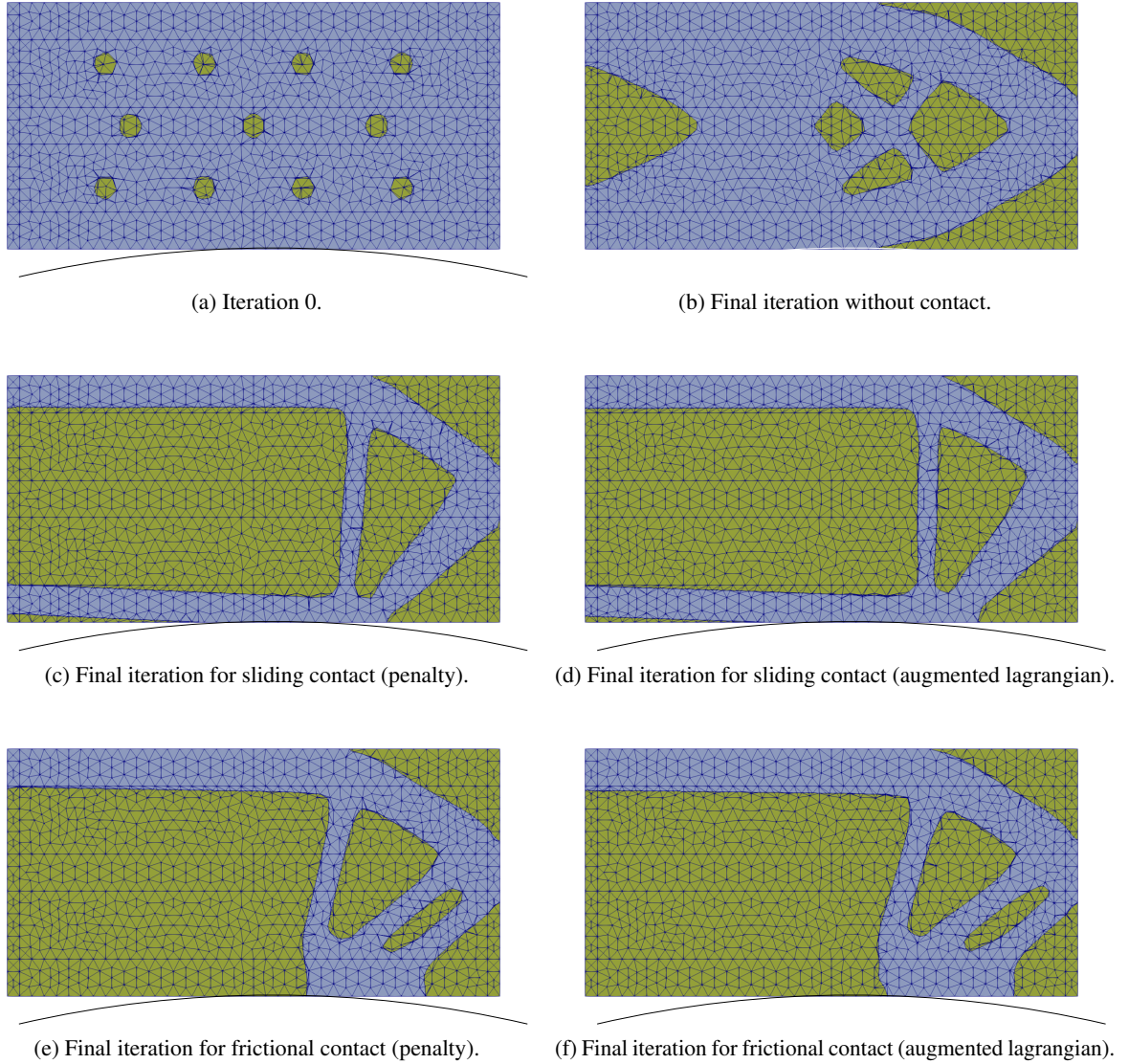
\begin{figure}[h]
  \begin{center}
    \subfloat[Iteration 0.]{
	\resizebox{0.46\textwidth}{!}{
	\begin{tikzpicture}
    	\node[anchor=south west,inner sep=0] at (0,0) {\includegraphics[width=\textwidth]{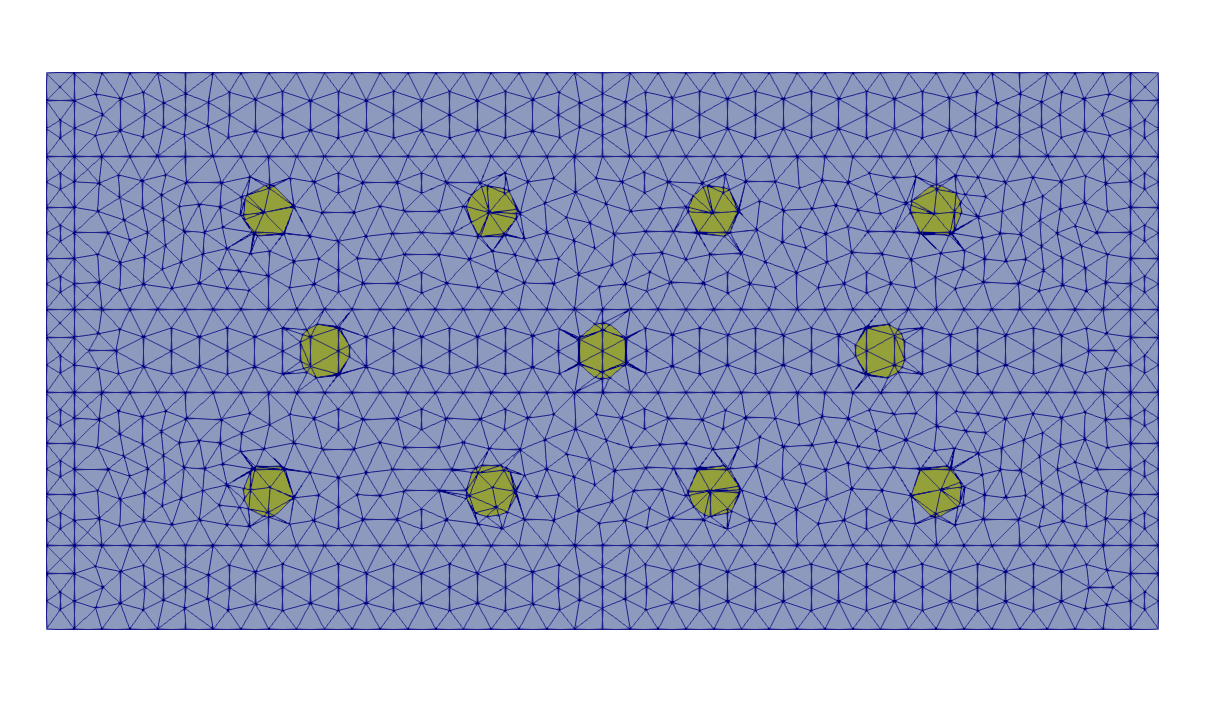}};
    	\draw[black] (16.75,0.13) arc (77:103:35);
    \end{tikzpicture}
    \label{sub:ResCanti2dCercleIt0}
    }
    }
    \hspace{.5em}
    \subfloat[Final iteration without contact.]{
	\resizebox{0.46\textwidth}{!}{
	\begin{tikzpicture}
    	\node[anchor=south west,inner sep=0] at (0,0) {\includegraphics[width=\textwidth]{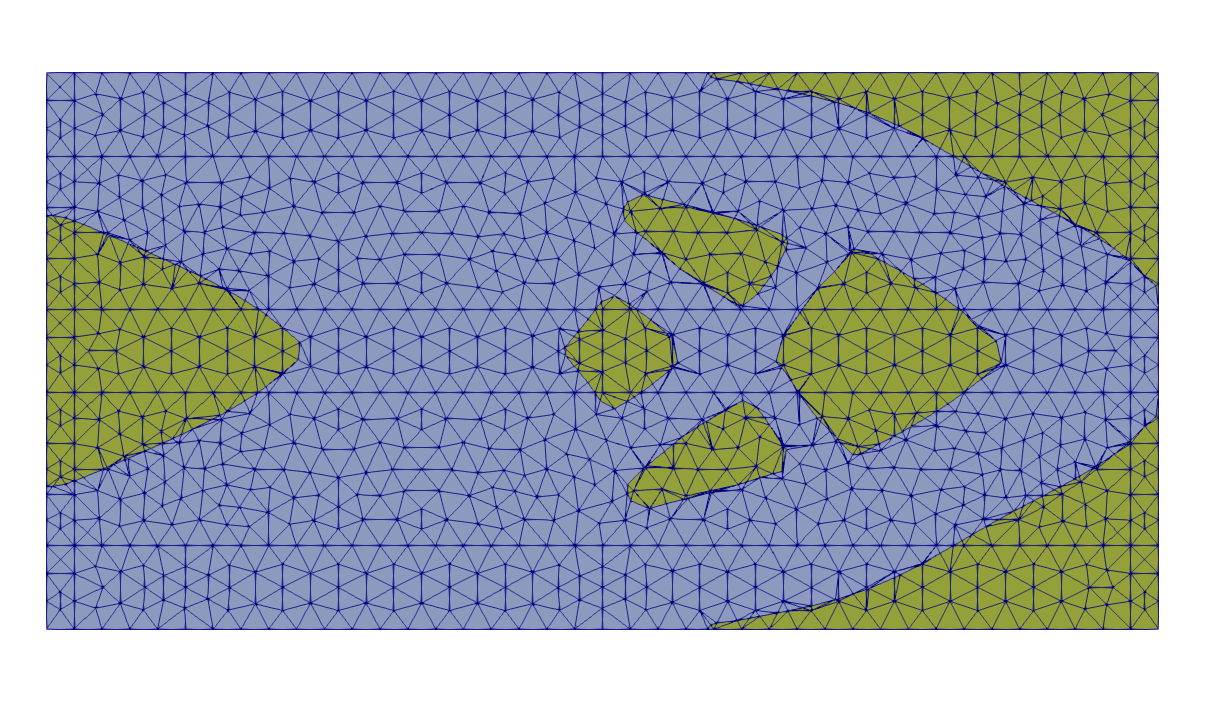}};
    	\draw[white] (16.75,0.13) arc (77:103:35);
    \end{tikzpicture}
    \label{sub:ResCanti2dBis}
    }
    }
    \hspace{.5em}
    \subfloat[Final iteration for sliding contact (penalty).]{
	\resizebox{0.46\textwidth}{!}{
	\begin{tikzpicture}
    	\node[anchor=south west,inner sep=0] at (0,0) {\includegraphics[width=\textwidth]{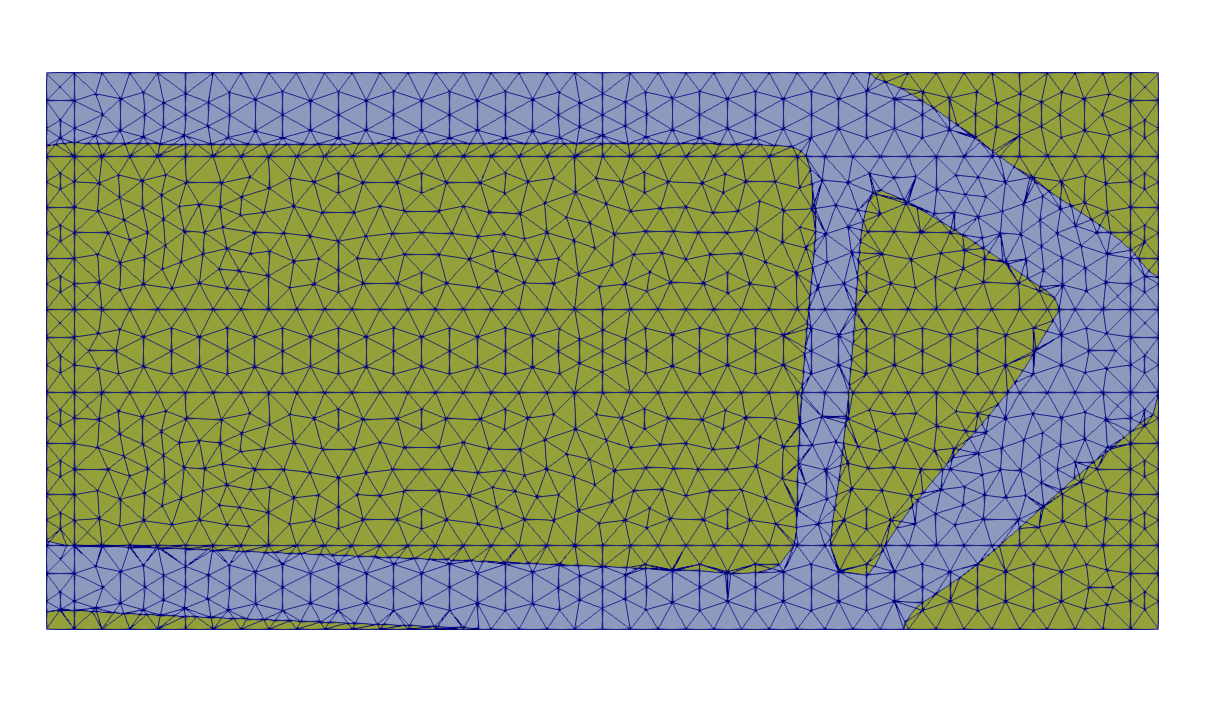}};
    	\draw[black] (16.75,0.13) arc (77:103:35);
    \end{tikzpicture}
    \label{sub:ResCanti2dCerclePena}
    }
    }
    \hspace{.5em}
    \subfloat[Final iteration for sliding contact (augmented lagrangian).]{
	\resizebox{0.46\textwidth}{!}{
	\begin{tikzpicture}
    	\node[anchor=south west,inner sep=0] at (0,0) {\includegraphics[width=\textwidth]{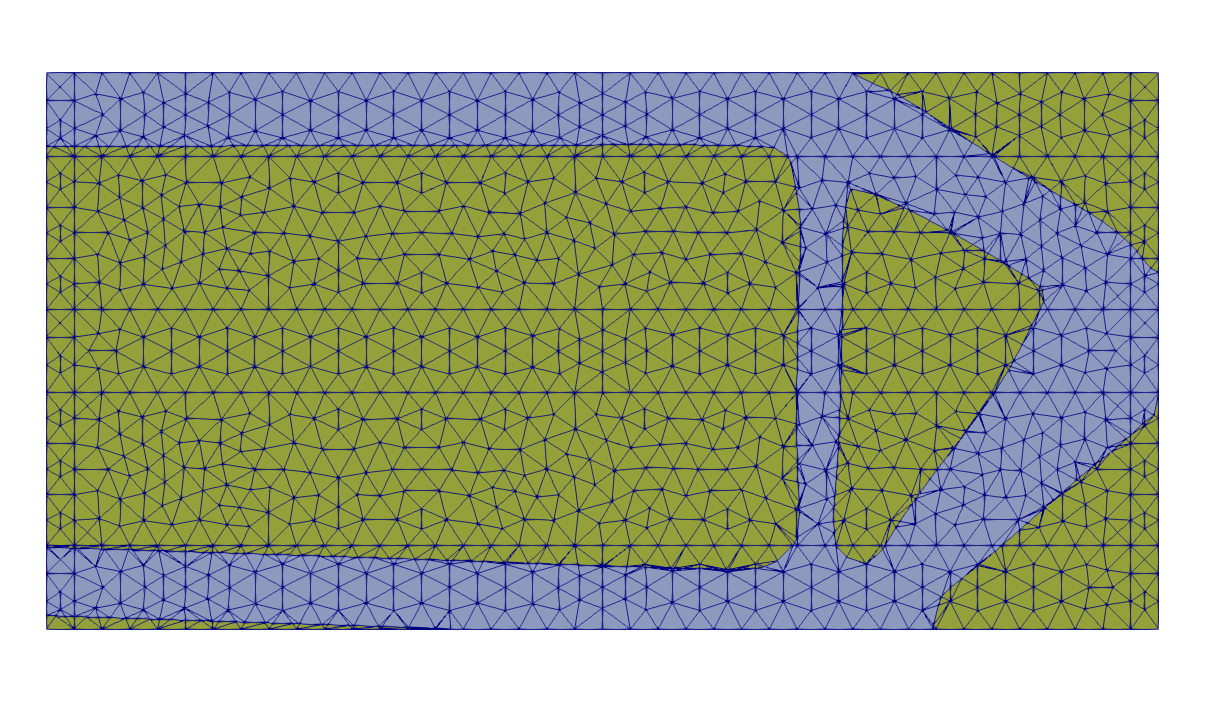}};
    	\draw[black] (16.75,0.13) arc (77:103:35);
    \end{tikzpicture}
    \label{sub:ResCanti2dCercleLagAug}
    }
    }
    \hspace{.5em}
    \subfloat[Final iteration for frictional contact (penalty).]{
	\resizebox{0.46\textwidth}{!}{
	\begin{tikzpicture}
    	\node[anchor=south west,inner sep=0] at (0,0) {\includegraphics[width=\textwidth]{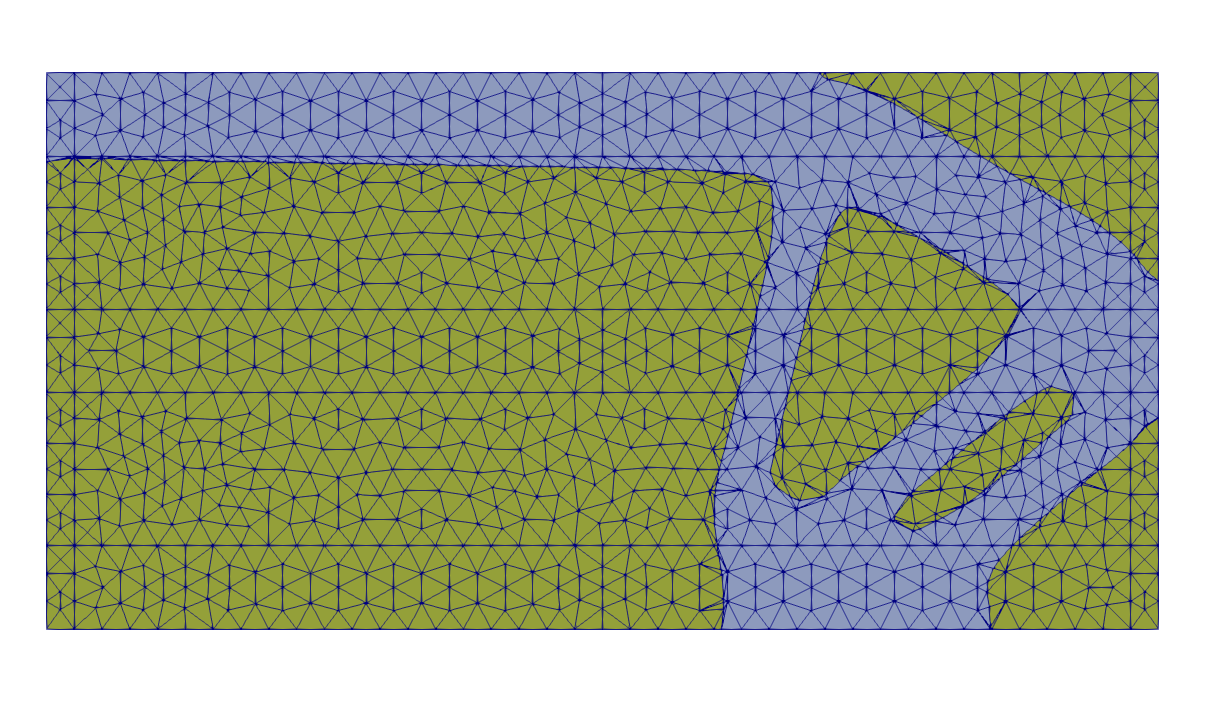}};
    	\draw[black] (16.75,0.13) arc (77:103:35);
    \end{tikzpicture}
    \label{sub:ResCanti2dCercleFrottPena}
    }
    }
    \hspace{.5em}
    \subfloat[Final iteration for frictional contact (augmented lagrangian).]{
	\resizebox{0.46\textwidth}{!}{
	\begin{tikzpicture}
    	\node[anchor=south west,inner sep=0] at (0,0) {\includegraphics[width=\textwidth]{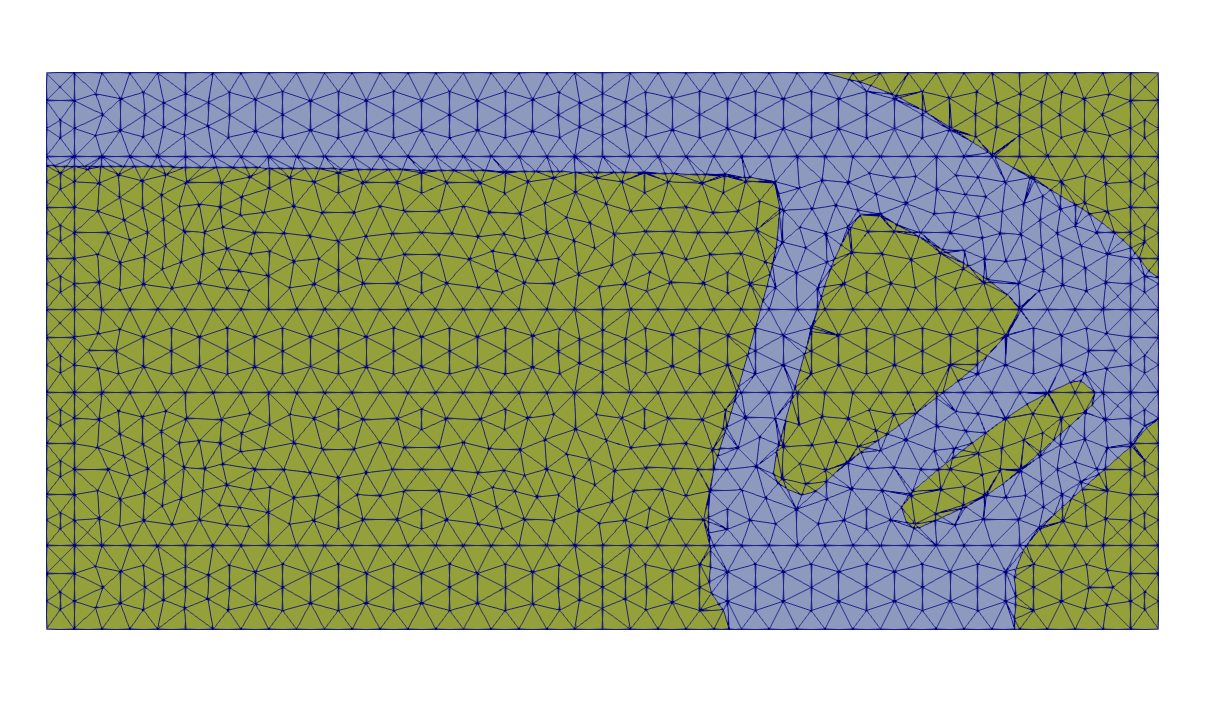}};
    	\draw[black] (16.75,0.13) arc (77:103:35);
    \end{tikzpicture}
    \label{sub:ResCanti2dCercleFrottLagAug}
    }
    }
    \end{center}
  \caption{Initial and final designs for the 2D cantilever in contact with a disk ($\Omega_h$ in blue, $D_h\setminus \Omega_h$ in yellow).}
  \label{fig:ResCanti2dCercle}
\end{figure}

\begin{figure}
\begin{center}
    \subfloat[Without contact.]{   
	\resizebox{!}{0.35\textwidth}{
	\begin{tikzpicture}
		\begin{axis}[
    		xlabel={Number of iterations $l$},
    		xmin=0, xmax=70,
    		]
    		\addplot[color=red,mark=none] table {res_canti2d_bis.txt};
   			\legend{$J(\Omega^l)$}
   		\end{axis}
	\end{tikzpicture}
	}
	}
	\hspace{2em}
    \subfloat[Sliding contact.]{   
	\resizebox{!}{0.35\textwidth}{
	\begin{tikzpicture}
		\begin{axis}[
    		xlabel={Number of iterations $l$},
    		xmin=0, xmax=310,
    		]
    		\addplot[color=gray,mark=none] table {res_canti2d_cercle_pena.txt};
    		\addplot[color=dartmouthgreen,mark=none] table {res_canti2d_cercle_lagaug.txt};
   			\legend{$J(\Omega^l)$ (penalty), $J(\Omega^l)$ (augmented lagrangian)}
   		\end{axis}
	\end{tikzpicture}
	}
	}
	\hspace{2em}
    \subfloat[Frictional contact.]{   
	\resizebox{!}{0.35\textwidth}{
	\begin{tikzpicture}
		\begin{axis}[
    		xlabel={Number of iterations $l$},
    		xmin=0, xmax=150,
    		]
    		\addplot[color=gray,mark=none] table {res_canti2d_cercle_frott_pena.txt};
    		\addplot[color=dartmouthgreen,mark=none] table {res_canti2d_cercle_frott_lagaug.txt};
   			\legend{$J(\Omega^l)$ (penalty), $J(\Omega^l)$ (augmented lagrangian)}
   		\end{axis}
	\end{tikzpicture}
	}
	}
\end{center}
\caption{Convergence curves for the 2D cantilever in contact with a disk.}
\label{fig:CvgIterCanti2dCercle}
\end{figure}
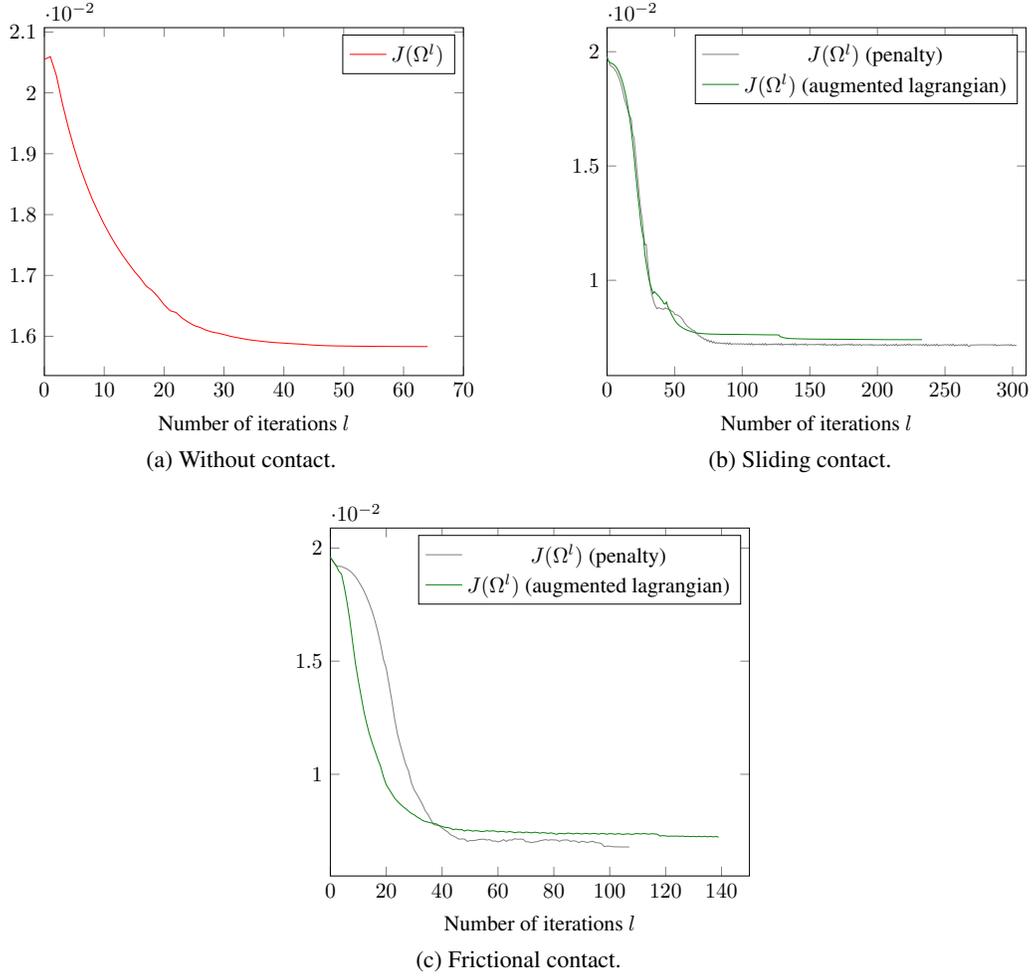
Again, we see that the penalty method and the augmented lagrangian method give very similar final designs in the cases of sliding and frictional contact. In each case, we notice that the location of the points in contact (initially at the center) has been moved to the right during the optimization process. Indeed, when the zone where $\Omega$ is supported by $\Omega_{rig}$ gets closer to $\Gamma_N$, the structure becomes more rigid under the action of the surface load $\tauu$. For the sliding contact model, the optimal shape has two anchor zones on $\hat{\Gamma}_D$: one at the top which improves rigidity, and another at the bottom which prevents the structure to slide down to the right. For the frictional contact model, there is only one anchor zone because the structure may now stick to the rigid foundation, which improves the stability near the contact zone.

Besides, as expected, if we compare the optimal designs obtained for the model with contact and the model without contact, we notice that thanks to the potential support offered by the rigid foundation in the case with contact, the algorithm manages to find structures which are as rigid as in the case without contact while being way lighter. This is confirmed by the values of $J$, see the tables in Figure \ref{fig:TabValeursJCanti2d}. Also note that adding friction enables us to slightly improve the value of $J$.

Even if the convergence histories of $J$ (see Figure \ref{fig:CvgIterCanti2dCercle}) reveal comparable behaviours for the penalty method and the augmented lagrangian method, the final values are quite different. More specifically, the values in the table of Figure \ref{fig:TabValeursJCanti2d} indicate that the optimal design obtained in the penalty case is $5\%$ more efficient than the one obtained in the augmented lagrangian case. In fact, it seems more likely that this difference comes from the inconsistency of the penalty method since the value of the penalty parameter in this benchmark is only $\epsilon = 10^5$.

\begin{figure}
\begin{center}
    \begin{tabular}{|c|c|c|}
    \hline
    \textbf{With contact} & sliding & frictional \\
    \hline
    Penalty & 0.00714989 & 0.00678781  \\
    \hline
    Augmented lagrangian & 0.00740285 & 0.00722132 \\
    \hline
    \end{tabular}
    \hspace{1.5em}
    \begin{tabular}{|c|c|}
    \hline
    \textbf{Without contact} & 0.0158314 \\
    \hline
    \end{tabular}
    \caption{Values of $J$ at the final iteration for the 2D cantilever.}
    \label{fig:TabValeursJCanti2d}
\end{center}
\end{figure}

\paragraph{Three-dimensional cantilever in contact with a ball.}

The three-dimensional case of optimizing the potential contact zone $\Gamma_C$ for a cantilever in contact with a flat rigid foundation has also been considered \cite{beremlijski2009shape}. In this reference, the authors use the much more cumbersome Coulomb friction model and, as in the two-dimensional case, $\Gamma_C$ is represented by a finite element function approximating the distance to the rigid plane. For the same reasons as in the 2D case, we suggest to consider a spherical rigid body $\Omega_{rig}$ instead of a plane. This benchmark can be found in contact mechanics literature, though a more popular variant is the case where the ball is replaced by a cylinder which axis is orthogonal to the axis of the beam.

The domain $D$ is a box with dimensions $5\times 3\times 2.4$, see Figure \ref{sub:MailCanti3dBoule}. The potential Dirichlet $\hat{\Gamma}_D$ is the whole left side and $\Gamma_N$ is a small rectangle at the center of the right side. Then, we add $\Omega_{rig}$, a ball of radius $R$ centered at $(2.5,-R,1.2)$. The external forces are still given by $\ff=0$ and $\tau=(0,0,-0.01)$, the Tresca threshold $s$ and the friction coefficient $\mathfrak{F}$ are taken as in the previous example, and the parameters related to the numerical resolution are such that $\varepsilon=10^7$ and $\gamma_1^k=\gamma_2^k=1000$, for all $k$. Finally, the coefficients of $J$ are set to $\alpha_1=30$, $\alpha_2=0.01$. In order to show the topology of the initial shape $\Omega$ at iteration 0, we have represented its complementary $D\setminus\Omega$ in Figure \ref{sub:ResCanti3dBouleIt0}.

\begin{figure}[h]
  \begin{center}
    \subfloat[Discrete domain $D_h$.]{
	\includegraphics[width=0.35\textwidth]{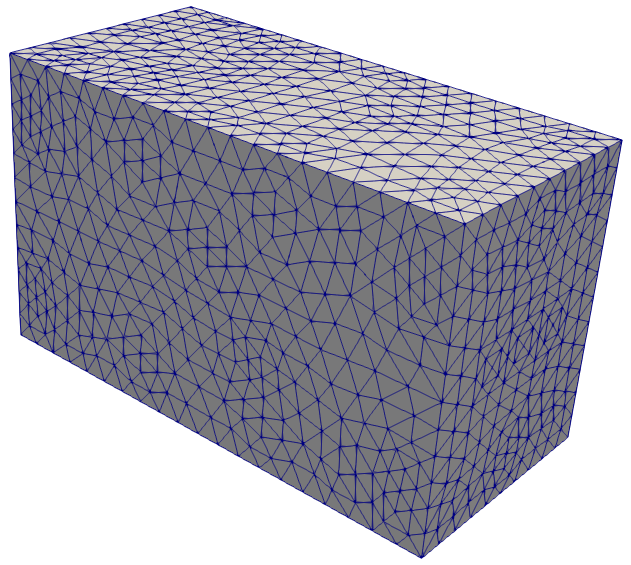}
    \label{sub:MailCanti3dBoule}
    }
    \hspace{.5em}
    \subfloat[Iteration 0.]{
	\includegraphics[width=0.35\textwidth]{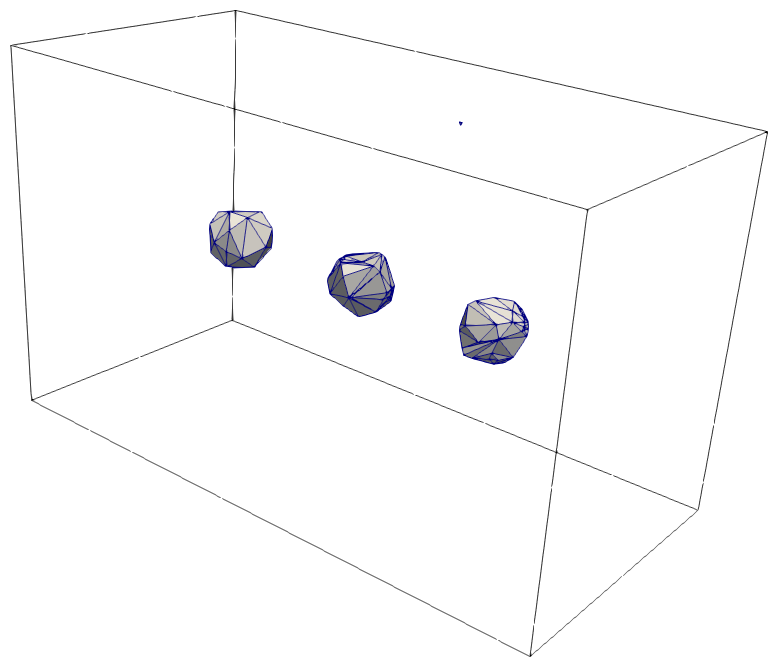}
    \label{sub:ResCanti3dBouleIt0}
    }
    \hspace{.5em}
    \subfloat[Final iteration for sliding contact (penalty).]{
	\includegraphics[width=0.35\textwidth]{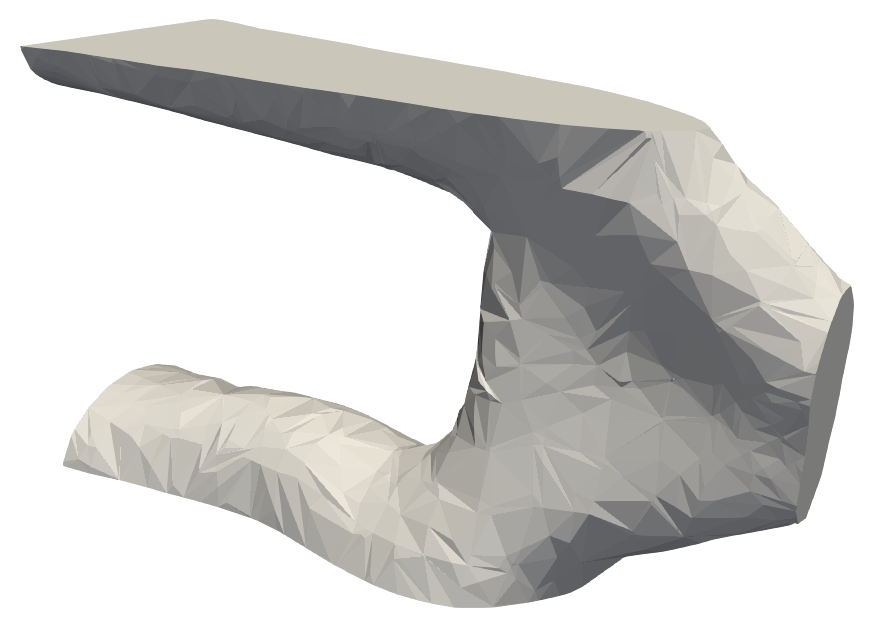}
    \label{sub:ResCanti3dBoulePena}
    }
    \hspace{.5em}
    \subfloat[Final iteration for sliding contact (augmented lagrangian).]{
	\includegraphics[width=0.35\textwidth]{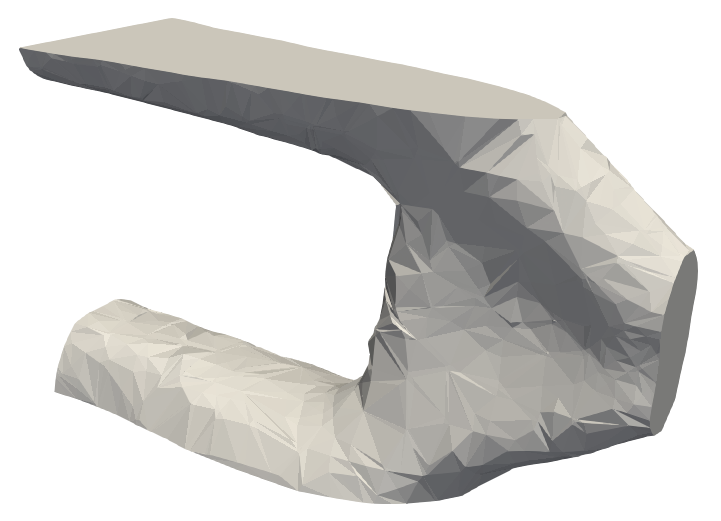}
    \label{sub:ResCanti3dBouleLagAug}
    }
    \hspace{.5em}
    \subfloat[Final iteration for frictional contact (penalty).]{
	\includegraphics[width=0.35\textwidth]{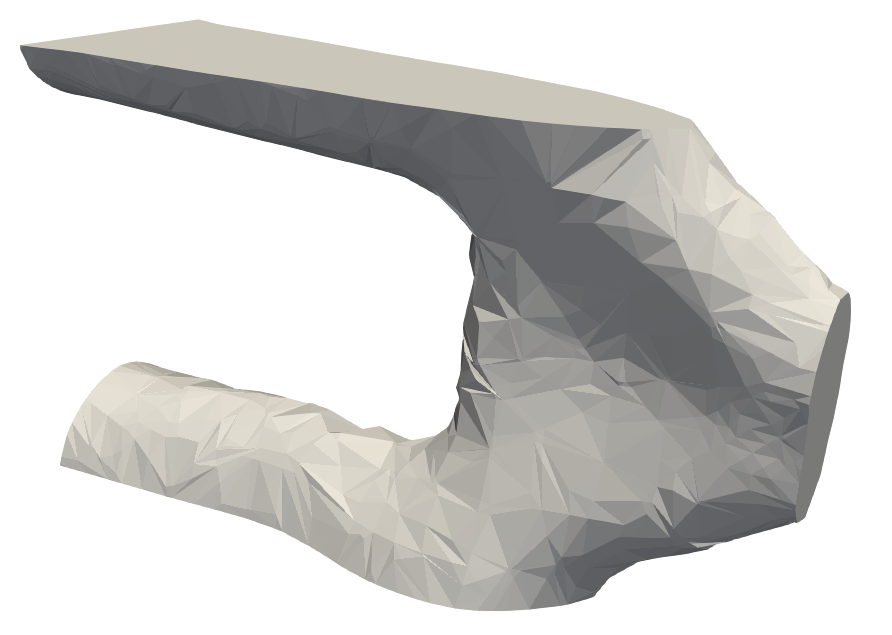}
    \label{sub:ResCanti3dBouleFrottPena}
    }
    \hspace{.5em}
    \subfloat[Final iteration for frictional contact (augmented lagrangian).]{
	\includegraphics[width=0.35\textwidth]{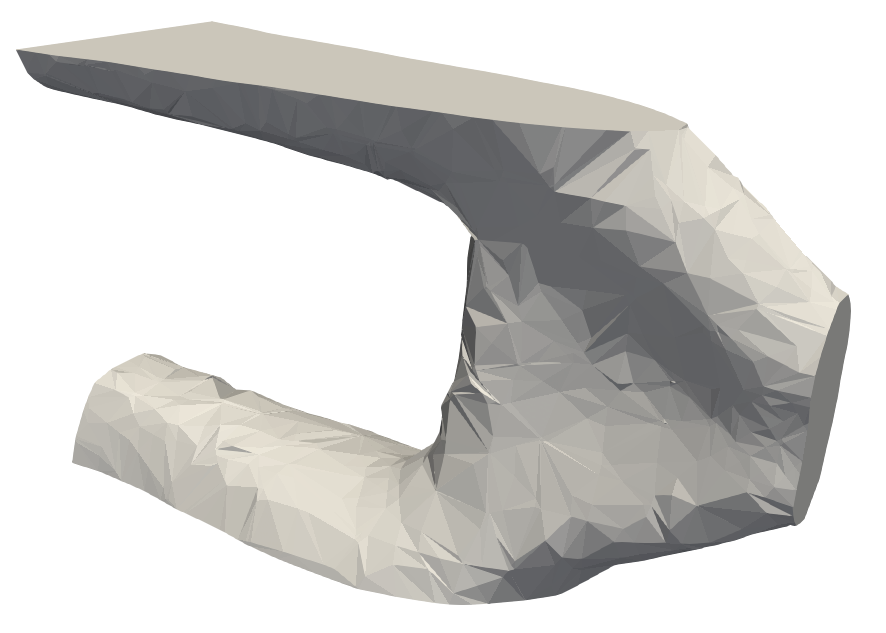}
    \label{sub:ResCanti3dBouleFrottLagAug}
    }
    \end{center}
  \caption{Initial and final designs for the 3D cantilever in contact with a ball.}
  \label{fig:ResCanti3dBoule}
\end{figure}

In this case as well, we obtain very similar optimal designs with the penalty and augmented lagrangian approaches, regardless of the contact model considered (sliding or frictional), see Figure \ref{fig:ResCanti3dBoule}. From the geometric point of view, the remarks made in te previous example remain valid in three dimensions: the zones where the optimal structures are supported by $\Omega_{rig}$ have been moved closer to $\Gamma_N$. As it is often the case in 3D shape optimization benchmarks, the dimension of structure in the rear-front direction has shrinked in order to make it as light as possible. Note that the coefficient $\alpha_1$ in front of the compliance in the expression of $J$ has been multiplied by two compared to the two-dimensional case, which explains why all optimal designs have two anchor zones, even in the frictional case.

Finally, when looking at the convergence history of $J$ in Figure \ref{fig:CvgIterCanti3dBoule}, one may notice a bump around iteration 30 in the case of frictional contact with the penalty method. This corresponds to iterations where the contact zone is moved a lot and thus the Newton solver for the mechanical problem is struggling. This is also the reason why we are limited in our choice of parameters: if we take smaller values of $\alpha_1$ or greater values of $\mathfrak{F}$, the shape might be subject to large displacements of the contact zone between two shape optimization iterations, which makes it very difficult for the Newton method to solve the penalty mechanical problem and eventually prevents the shape optimization algorithm from converging. In practice, we have not encountered similar difficulties with the augmented lagrangian method. In addition, the convergence of $J$ for the augmented lagrangian method seems smoother than the one for the penalty method in both sliding and frictional cases, which suggests that the augmented lagrangian method is more stable from the shape optimization point of view.

\begin{figure}
\begin{center}
    \subfloat[Sliding contact.]{   
	\resizebox{!}{0.35\textwidth}{
	\begin{tikzpicture}
		\begin{axis}[
    		xlabel={Number of iterations $l$},
    		xmin=0, xmax=170,
    		]
    		\addplot[color=gray,mark=none] table {res_canti3d_sphere_pena.txt};
    		\addplot[color=dartmouthgreen,mark=none] table {res_canti3d_sphere_lagaug_bis.txt};
   			\legend{$J(\Omega^l)$ (penalty), $J(\Omega^l)$ (augmented lagrangian)}
   		\end{axis}
	\end{tikzpicture}
	}
	}
	\hspace{2em}
    \subfloat[Frictional contact.]{   
	\resizebox{!}{0.35\textwidth}{
	\begin{tikzpicture}
		\begin{axis}[
    		xlabel={Number of iterations $l$},
    		xmin=0, xmax=160,
    		]
    		\addplot[color=gray,mark=none] table {res_canti3d_sphere_frott_pena.txt};
    		\addplot[color=dartmouthgreen,mark=none] table {res_canti3d_sphere_frott_lagaug.txt};
   			\legend{$J(\Omega^l)$ (penalty), $J(\Omega^l)$ (augmented lagrangian)}
   		\end{axis}
	\end{tikzpicture}
	}
	}
\end{center}
\caption{Convergence curves for the 3D cantilever in contact with a ball.}
\label{fig:CvgIterCanti3dBoule}
\end{figure}
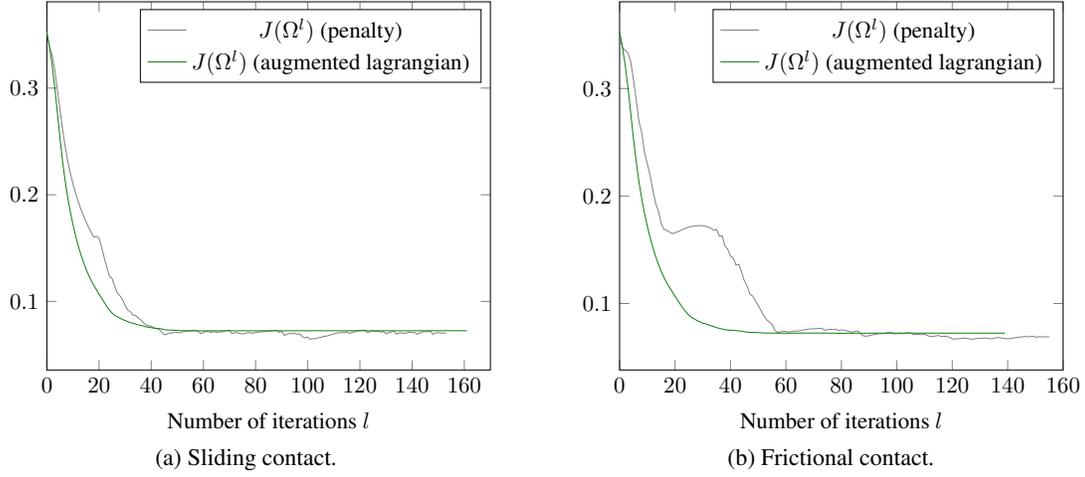

\section{Conclusion}

This article presents a topology optimization algorithm based on directional shape derivatives and a level-set approach in the context of three-dimensional contact mechanics with Tresca friction. Two regularized formulations of the contact problem have been studied: the penalty formulation and the augmented lagrangian formulation. In both cases, the expressions of the shape derivatives have been obtained at the continuous level, which makes them independent of the discretization chosen for the mechanical problem. In addition, the algorithm benefits from a simple but original mesh-cutting technique based on a quadratic interpolation of the level-set function. This enables us to have a mesh of the discrete domain at each iteration and thus apply the contact conditions exactly on $\Gamma_C$. Finally, the validity of our approach has been illustrated in three numerical examples. The results suggest that the augmented lagrangian method is more stable and robust than the penalty method in the context of shape optimization, especially in three-dimensions.

A natural extension of this work would be to consider the more realistic Coulomb friction model. Of course, there is a huge gap in terms of theoretical and numerical difficulties between the Tresca model and the Coulomb model. However, it is known \cite{eck2005unilateral} that under some specific assumptions, the solution to the Coulomb problem can be approximated by a fixed point method which solves a Tresca problem at each iteration. From the numerical point of view, it might be necessary in this case to include a post-processing remeshing procedure in our algorithm in order to improve the quality of the mesh after cutting it.

\section*{Acknowledgments}
This work was supported by the Natural Sciences and Engineering Research Council of Canada (NSERC).

\bibliographystyle{acm}
\bibliography{bibliography}

\end{document}